\documentclass[oneside,english]{amsart}
\usepackage[T1]{fontenc}
\usepackage[latin9]{inputenc}
\usepackage{geometry}
\usepackage{babel}
\usepackage{verbatim}
\usepackage{float}
\usepackage{mathrsfs}
\usepackage{amsthm}
\usepackage{amstext}
\usepackage{amssymb}
\usepackage{graphicx}
\usepackage[unicode=true,
 bookmarks=true,bookmarksnumbered=false,bookmarksopen=false,
 breaklinks=false,pdfborder={0 0 1},backref=false,colorlinks=false]
 {hyperref}

\makeatletter

\floatstyle{ruled}
\newfloat{algorithm}{tbp}{loa}
\providecommand{\algorithmname}{Algorithm}
\floatname{algorithm}{\protect\algorithmname}

\theoremstyle{plain}
\newtheorem{thm}{\protect\theoremname}
  \theoremstyle{definition}
  \newtheorem{defn}[thm]{\protect\definitionname}
  \theoremstyle{remark}
  \newtheorem{rem}[thm]{\protect\remarkname}
  \theoremstyle{plain}
  \newtheorem{lem}[thm]{\protect\lemmaname}

\usepackage{helvet}
\usepackage{courier}

\renewcommand{\vec}{\boldsymbol}
\newcommand{\tp}{'}

\makeatother

  \providecommand{\definitionname}{Definition}
  \providecommand{\lemmaname}{Lemma}
  \providecommand{\remarkname}{Remark}
\providecommand{\theoremname}{Theorem}

\begin{document}
\title[Stochastic Approximation for Quasi-Stationary Distributions in Finite Dimensions]{Theoretical Analysis of a Stochastic Approximation Approach for Computing Quasi-Stationary Distributions}

\author{Jose Blanchet, Peter Glynn, Shuheng Zheng}
\begin{abstract}
This paper studies a method, which has been proposed in the Physics
literature by \cite{latestoliveira,oliveiradickman1,dickman}, for
estimating the quasi-stationary distribution. In contrast to existing
methods in eigenvector estimation, the method eliminates the need
for explicit transition matrix manipulation to extract the principal
eigenvector. Our paper analyzes the algorithm by casting it as a stochastic
approximation algorithm (Robbins-Monro) \cite{robbins1951stochastic,kushner2003stochastic}.
In doing so, we prove its convergence and obtain its rate of convergence.
Based on this insight, we also give an example where the rate of convergence
is very slow. This problem can be alleviated by using an improved
version of the algorithm that is given in this paper. Numerical experiments
are described that demonstrate the effectiveness of this improved
method.
\end{abstract}
\maketitle

\section{Introduction}

The motivation for this algorithm came from physicists' need to estimate
the quasi-stationary distribution of interacting particle systems
(IPS) \cite{latestoliveira,oliveiradickman1,dickman,liggett_ips}.
A quasi-stationary distribution can be computed via the left principal
eigenvector of the substochastic transition matrix over the non-absorbing
states (transition rate matrix in continuous-time). However, the eigenvalue
problem suffers from the curse of dimensionality, and is especially
prohibitive in IPS where the state space is very large. 

A sampling based method has been proposed by these physicists based
on a heuristic manipulation of the Kolmogorov forward equation. The
validity of this method actually has been a priori proven by \cite{aldous1988two,athreya,pemantle}
who casted it as a generalized urn process. Results on rates of convergence
have been obtained. This result involves a Central Limit Theorem (CLT)
for a specific set of functionals corresponding to non-principal eigenvectors
of the underlying substochastic matrix. 

Our main contribution are as follows
\begin{enumerate}
\item Our paper recognizes the algorithm as a stochastic approximation algorithm
(Section \ref{sub:Formal-Description}).
\item This allows us to prove convergence and sufficient conditions for
a stronger CLT (Theorem \ref{thm:DTMC-unprojected}) that is not restricted
only to specific functionals.
\item More importantly, we recognized common scenarios (Section \ref{sub:Counter-example})
where the CLT fails and significantly hamper the performance of the
algorithm (i.e. very slow rate of convergence).
\item At the end, we came up with an improved algorithm (Section \ref{sub:Projection-Algorithm})
which exhibits a valid CLT under all scenarios by using projection
and iterate averaging \cite{polyak1992acceleration}.
\end{enumerate}
Section \ref{sec:Background} reviews some background material to
the contact process, quasi-stationary distributions, mentions a less
powerful method of proof via urn processes, and reviews the relevant
related literature on eigenvector estimations and points out their
shortcomings. Section \ref{sec:Motivation} explains the the basis
for the original heuristic and outlines the algorithm. Section \ref{sec:analysis}
goes over the stochastic approximation formulation and sketches the
proof of convergence (the full proof is given in the Appendix \ref{sub:Proof-of-DT}).
Section \ref{sec:Variations} gives an improved version of the algorithm
using projection along with its faster rate of convergence result.
Section \ref{sec:CTMC} briefly studies the algorithm adapted for
continuous-time Markov chains. Section \ref{sec:NumericalExp.} goes
over several important numerical experiments.

\section{Background and Related Literature\label{sec:Background}}

\subsection{Quasi-Stationary Distribution}

\subsubsection{Discrete-Time Version}

The paper \cite{darroch1965quasi} proposed the concepts of quasi-stationary
distribution and quasi-limiting distribution for discrete-time Markov
chains. Assume that $0$ is the absorbing state and $1,\ldots,n$
are non-absorbing, we can partition the Markov transition matrix as
\[
P=\left[\begin{array}{cc}
1 & 0\\
\vec{\alpha} & Q
\end{array}\right]
\]
where $Q$ is a n-by-n matrix. 

First we define the conditional transition probabilities
\begin{eqnarray*}
d_{j}^{\vec{\pi}}(n) & = & \mathbb{P}(X_{n}=j|X_{0}\sim\pi,X_{1},\ldots X_{n-1}\neq0)\\
 & = & \frac{\vec{\pi}^{\tp}Q^{n}\vec{e}_{j}}{\vec{\pi}^{\tp}Q^{n}\vec{e}}
\end{eqnarray*}
where $\{\vec{e}_{i}\}$ is the standard basis for $\mathbb{R}^{n}$,
$\vec{\pi}$ is a probability distribution, and $\vec{e}$ is the
vector of all 1's. $\vec{d}^{\vec{\pi}}(n)$ is the vector whose j-th
component is $d_{j}^{\vec{\pi}}(n)$. This leads to the following
definition.
\begin{defn}
If there is a distribution $\vec{\pi}$ over the transient states
such that $\vec{d}^{\vec{\pi}}(n)$ is independent of $n$, then we
call $\vec{d}^{\vec{\pi}}$ the quasi-stationary distribution.
\end{defn}
Under the assumption that the substochastic matrix $Q$ is irreducible
(though not necessarily aperiodic), it is straightforward to see that
the quasi-stationary distribution exists and is the unique solution
to principal eigenvector problem
\[
\vec{d}^{\tp}Q=\rho\vec{d}^{\tp}
\]
This existence and uniqueness (assuming that $\vec{d}$ is normalized
to be a probability vector) can be obtained by the Perron-Frobenius
theorem \cite{karlin_taylor}.

The paper \cite{meleard2011quasi} explores the existence of quasi-stationary
distribution for countable and general state space Markov chains where
$Q$ is replaced with the generator and $\vec{d}$ is a measure.

\subsubsection{Continuous-Time}

If we think about the transition rate matrix of a CTMC under similar
setup (irreducibility), then it (\cite{darroch1967ctmc}) can be said
that 
\[
d_{j}^{\vec{\pi}}(t)\rightarrow d_{j}+o(e^{t(\rho'-\rho_{1})})
\]
where $\vec{d}$ is the principal left-eigenvector of the rate matrix
corresponding to the transient states with associated eigenvalue $\rho_{1}$,
i.e.
\[
\vec{d}^{\tp}R=\rho_{1}\vec{d}^{\tp}
\]
where $R$ is the rate matrix of the CTMC.

\subsection{Linear Algebra Methods}

Classical linear algebra methods such as the power method \cite{golubvanloan}
suffers from the curse of dimensionality. Monte Carlo power methods
by \cite{dimov1998monte} can be adapted to produce eigenvectors but
requires an explicit computation of the substochastic transition matrix
on the fly and is expensive to do for interacting particle systems.
There exists a set of stochastic approximation methods for determining
principal eigenvalue/eigenvector where the matrix is random; however,
it too requires explicit matrix multiplication \cite{krasulina1970,krasulina1969,oja1985stochastic}
which is infeasible in interacting particle system case. 

Interacting particle systems such as the contact process (Section
\ref{sub:Contact-Process}) suffers heavily from the curse of dimensionality
and renders existing classical methods infeasible. This is an important
class of problems for physicists\cite{dickman,latestoliveira,marro,oliveiradickman1}
in the study of phase-transition property of certain non-equilibrium
systems. 

Lastly, a large number of adaptive algorithms has been designed for
estimating principal eigenvector of only covariance matrices (positive
semi-definite) where you observe an i.i.d. sequence of random vectors
with that particular covariance matrix \cite{chatterjee}.

\subsection{Fleming-Viot method}

The Fleming-Viot method \cite{meleard2011quasi,ferrari,burdzy} is
an interacting particle system that allows us to compute quasi-stationary
distributions of countable Markov chains and diffusion processes.
It consists of $N$ particles evolving independently according to
the dynamics of the Markov process. If one particle gets absorbed,
it is immediately restarted from a position uniformly picked from
the remaining $N-1$ particles. As both time and $N$ goes to infinity,
this would converge to the quasi-stationary distribution. When the
state space is large, you need $N$ to be large enough to have a good
approximation to the true quasi-stationary distribution. This would
be prohibitive in interacting particle systems. Furthermore, it is
generally computationally expensive to increase both the number of
particles $N$ and the time of the simulation.

\subsection{Urn Processes}

The algorithm which will be described below has been previously analyzed
as a generalized Polya's urn \cite{athreya,aldous1988two}. The overview
paper \cite{pemantle} is a comprehensive survey of urn processes.
However, the rate of convergence results of these urn processes are
not as strong as our result. They only offer a CLT along the non-principal
right eigenvectors of the rate matrix whereas we offer a CLT along
every direction in the space. The set of non-principal right eigenvectors
can never span the whole space (which in this case can be taken to
be the hyperplane orthogonal to $\vec{1}$) because the principal
left eigenvector $\bar{\vec{\mu}}$ is orthogonal to all the non-principal
right eigenvectors. Unless if $\bar{\vec{\mu}}$ is a multiple of
$\vec{1}$, we may only examine the CLT along the space orthogonal
to $\bar{\vec{\mu}}$ as opposed to the whole hyperplane that is orthogonal
to $\vec{1}$. In summary, our results are a strict extension of the
available corresponding results on urn processes. However, more importantly,
our approach is fundamentally different and builds on the well-studied
machinery of stochastic approximations and therefore allows us to
obtain significant algorithmic improvements that we shall explain
(Theorems \ref{thm:projection-thm} and \ref{thm:polyak-thm}) and
experimentally demonstrate (Section \ref{sec:NumericalExp.}).

\section{Heuristic Motivation\label{sec:Motivation}}

\subsection{Motivation from the Physics Literature}

This section reviews the heuristic origin of the algorithm from the
physics literature \cite{dickman,latestoliveira,oliveiradickman1}.
Under the setting of a continuous-time Markov chain with rate matrix
$R$ and absorbing state $0$ (without loss of generality, we can
combine all absorbing states into one state), if we define $p_{ij}(t)=P(X_{t}=j|X_{0}=i)$
and $P_{is}(t)=1-p_{i0}(t)$ , then we have that the quasi-stationary
distribution $d_{j}=\lim_{t\rightarrow\infty}\frac{p_{ij}(t)}{P_{is}(t)}$.
If we apply the Kolmogorov forward equation (known to physicists as
the master equation), we get that
\begin{equation}
\frac{dp_{ij}(t)}{dt}=\sum_{k}p_{ik}(t)R_{kj}\label{eq:1}
\end{equation}
and
\begin{equation}
\frac{dP_{is}(t)}{dt}=\frac{d}{dt}(1-p_{i0}(t))=-\sum_{k}p_{ik}(t)R_{k0}.\label{eq:2}
\end{equation}
Intuitively by the definition of $d_{j}$, we have that $p_{ij}(t)\approx d_{j}P_{is}(t)$
in the quasi-stationary time window ($t$ large enough). So we can
apply this to the preceding two equations and get 
\begin{eqnarray*}
d_{j}\left(\frac{dP_{is}(t)}{dt}\right) & = & \sum d_{k}P_{is}(t)R_{kj}\\
\frac{dP_{is}(t)}{dt} & = & -\sum_{k}d_{k}P_{is}(t)R_{k0}.
\end{eqnarray*}
Combine the two and we get
\[
d_{j}(\sum_{k}d_{k}R_{k0})+\sum_{k}d_{k}R_{kj}=0.
\]
This gives us a non-linear equation for the equilibrium condition
for the quasi-stationary distribution $\vec{d}$. We can think of
this as the stationary point of the forward equation
\begin{equation}
\frac{d(d_{j})}{dt}=\sum_{k}d_{k}R_{kj}+d_{j}(\sum_{k}d_{k}R_{k0}).\label{eq:nonlinear_mastereqn}
\end{equation}
The first part of this equation is the standard Kolmogorov forward
equation, while the second part redeposits the probability of hitting
the absorbing states onto all the non-absorbing states according to
the current distribution $d_{j}$. 

This previous discussion suggests the following algorithm:

\begin{algorithm}[H]
\begin{enumerate}
\item Initialize a vector $\vec{\mu}=0$ with dimension equal to the number
of non-absorbing states in the Markov chain. (Each component represents
the total number of visits to the corresponding non-absorbing state.)
\item Select any non-absorbing state of the chain, say state i and let $X_{0}=i$
\item Simulate the Markov chain starting from state $X_{0}$ up until absorption
and update $\vec{\mu}$ by adding the number of visits to each state
until absorption (so, for example, the number of visits to i is at
least one).
\item Select a non-absorbing state according to the normalized vector $\vec{\mu}$
(so that it becomes a probability vector). Let such non-absorbing
state be $X_{0}$ and go to Step 3.
\item Repeat Steps 3) and 4) many times and output the normalized vector
$\vec{\mu}$ as your approximation of the quasi-stationary distribution.
You can also output the averaged time to absorption in each tour as
an approximation to $\frac{1}{1-\lambda}$ where $\lambda$ is the
principal eigenvalue of the transition (rate) matrix.
\end{enumerate}
\caption{Algorithm for estimating quasi-stationary distribution}
\label{enu:mainalgo}
\end{algorithm}

For large enough time, the dynamics of the Markov chain will be governed
by Equation (\ref{eq:nonlinear_mastereqn}), which means we can obtain
the quasi-stationary distribution by examining the empirical distribution
after some large enough time.

\section{Stochastic Approximation Analysis of the Algorithm\label{sec:analysis}}

In this section, we will cast Algorithm \ref{enu:mainalgo} into a
stochastic approximation algorithm. This will let us rigorously prove
convergence and CLT for the algorithm.

\subsection{Brief Review of Stochastic Approximation and Intuition}

Consider the root-finding task of finding $\vec{\theta}$ such that
$f(\vec{\theta})=0$ with the restriction that only access to ``noisy''
observations of $f(\vec{\theta})$, denoted by $F(\vec{\theta})$,
are available. If $f$ is suitably smooth and the root is simple enough,
we can consider the descent method given by
\[
\vec{\theta}_{n+1}=\vec{\theta}_{n}+\epsilon_{n}F(\vec{\theta}_{n})
\]
where $\epsilon_{n}$ is a positive sequence going to zero.

We need to rigorously define the type of noise on the function $F$.
There are several conditions on noise which, when imposed, will lead
to convergence guarantees on $\{\vec{\theta}_{n}\}$. Here we focus
on the simplest \emph{martingale difference }noise type. Let the n-th
observation be denoted by $\vec{W}_{n}$, which could theoretically
depend on the whole history $\mathscr{F}_{n}=\sigma\{\vec{\theta}_{k},\vec{W}_{k-1}|1\leq k\leq n\}$.
In that case the descent method is written as 
\[
\vec{\theta}_{n+1}=\vec{\theta}_{n}+\epsilon_{n}\vec{W}_{n}
\]
The martingale difference noise requires that there exists a $g$
that satisfies
\[
\mathbb{E}[\vec{W}_{n}|\mathscr{F}_{n}]=\vec{g}(\vec{\theta}_{n}).
\]

If we impose the step-size condition $\sum\epsilon_{n}^{2}<\infty$
then $\vec{\theta}_{n}\approx\vec{\theta}_{0}+\sum_{k=1}^{n-1}\epsilon_{k}\vec{W}_{k}$.
The variance of the last sum is a finite number. Thus if we rewrite
the recursion as $\frac{\vec{\theta}_{n+1}-\vec{\theta}_{n}}{\epsilon_{n}}=\vec{W}_{n}$
and impose the condition $\epsilon_{n}\downarrow0$, heuristically
we predict that $\vec{\theta}_{n}$ should be related to the stationary
points of the ODE
\[
\dot{\vec{\theta}}(t)=\vec{g}(\vec{\theta}(t)).
\]

Furthermore, if we impose $\sum\epsilon_{n}=\infty$, we know that
in some sense, $\vec{\theta}_{n}$ would move by $\sum\epsilon_{n}=\infty$
steps and should converge to the stable attractors (either orbits
or points) of this ODE (\cite{kushner2003stochastic} Theorem 5.2.1).

\subsection{Precise Description of the Algorithm\label{sub:Formal-Description}}

We will now write down a precise description of the above heuristic
Algorithm \ref{enu:mainalgo} and convert it into stochastic approximation
form. 

Notation
\begin{itemize}
\item $S$ is the state space of the Markov chain whose quasi-stationary
distribution we are trying to estimate.
\item $T\subsetneq S$ is the set of transient states of the Markov chain
\item $Q$ is the substochastic matrix over the transient states $T$.
\item $\vec{\mu}_{n}$ will be a sequence of probability vectors over the
transient states $T$. This vector will store the cumulative empirical
distribution up to, and including, the n-th iteration of the algorithm.
$\vec{\mu}_{n}(x)$ is its value at a particular transient state $x$.
\item $\{X_{k}^{(n)}\}_{k}$ will be the Markov chain used in the n-th iteration
of the algorithm. They're independent conditioned on the initial distribution.
The n-th Markov chain will have initial distribution $\vec{\mu}_{n}$.
\item $\tau^{(n)}=\min\{k\geq0|X_{k}^{(n)}\not\notin T\}$. The hitting
time of the absorbing state of the n-th iteration
\end{itemize}
We can write Algorithm \ref{enu:mainalgo} as a recursion
\[
\vec{\mu}_{n+1}(x)=\frac{\left(\sum_{k=0}^{n}\tau^{(k)}\right)\vec{\mu}_{n}(x)+\left(\sum_{k=0}^{\tau^{(n+1)}-1}\mathbb{I}(X_{k}^{(n+1)}=x|X_{0}^{(n+1)}\sim\vec{\mu}_{n})\right)}{\sum_{k=0}^{n+1}\tau^{(k)}}\quad\forall x\in T
\]
where we can take the first probability vector $\vec{\mu}_{0}$ arbitrarily.

We will transform $\vec{\mu}_{n}$ into stochastic approximation form
by re-factoring:
\begin{eqnarray*}
\vec{\mu}_{n+1}(x) & = & \vec{\mu}_{n}(x)+\left(\frac{1}{n+1}\right)\left(\frac{\sum_{l=0}^{\tau^{(n+1)}-1}\left(\mathbb{I}(X_{l}^{(n+1)}=x)-\vec{\mu}_{n}(x)\right)}{\frac{1}{n+1}\sum_{j=0}^{n+1}\tau^{(j)}}\right).
\end{eqnarray*}
The denominator is problematic because its conditional expectation
(on $\mathscr{F}_{n})$ is not only a function of $\vec{\mu}_{n}$
but depends on the whole history of $\vec{\mu}_{n}$. To solve this,
we artificially add another state $T_{n}$ in the following way.

\begin{equation}
\begin{cases}
T_{n+1} & =T_{n}+\frac{1}{n+2}(\tau^{(n+1)}-T_{n})\quad\textrm{equivalent to}\quad T_{n}=\frac{1}{n+1}\sum_{j=0}^{n}\tau^{(j)}\\
\vec{\mu}_{n+1}(x) & =\vec{\mu}_{n}(x)+\\
 & \qquad\left(\frac{1}{n+1}\right)\left(\frac{\sum_{l=0}^{\tau^{(n+1)}-1}\left(\mathbb{I}(X_{l}^{(n+1)}=x|X_{0}^{(n+1)}\sim\vec{\mu}_{n})-\vec{\mu}_{n}(x)\right)}{T_{n}+\frac{\tau^{(n+1)}}{n+1}}\right).
\end{cases}\label{eq:mainalg_formal}
\end{equation}
We can therefore define
\begin{eqnarray}
Y_{n}(\vec{\mu},T)(x) & \triangleq & \frac{\sum_{l=0}^{\tau-1}\left(\mathbb{I}(X_{l}=x|X_{0}\sim\vec{\mu})-\vec{\mu}(x)\right)}{T+\frac{\tau}{n+1}}\label{eq:iterates_defined}\\
Z(\vec{\mu},T) & \triangleq & (\tau-T)\quad\textrm{where the initial distribution is \ensuremath{\vec{\mu}}},\nonumber 
\end{eqnarray}
and rewrite the stochastic approximation recursion as

\[
\begin{cases}
\vec{\mu}_{n+1}(x) & =\vec{\mu}_{n}(x)+\left(\frac{1}{n+1}\right)\vec{Y}_{n}(\vec{\mu}_{n},T_{n})(x)\\
T_{n+1} & =T_{n}+\left(\frac{1}{n+2}\right)Z(\vec{\mu}_{n},T_{n}).
\end{cases}
\]

\begin{rem}
Note:
\begin{itemize}
\item The term $\vec{Y}_{n}$ has an explicit dependence on $n$. That is
fine as that portion is asymptotically negligible. The details are
in the Appendix \ref{sub:Proof-of-DT}.
\item Please note that the iterates $\vec{\mu}_{n}$ are constrained in
$H\triangleq\{\vec{x}\in\mathbb{R}_{+}^{n}|\sum x_{i}=1\}$ (check
by inner producting with $\vec{1}$). This way we automatically satisfy
the boundedness requirement in \cite{kushner2003stochastic}.
\item We can also define a similar algorithm for the continuous-time Markov
chain by keeping track of the amount of time a Markov chain spends
in each transient state. This is given in Section \ref{sec:CTMC}.
\end{itemize}
\end{rem}

\subsection{Convergence}

The main result in this section can now be stated.
\begin{thm}
\label{thm:DTMC-unprojected}Given an irreducible absorbing Markov
chain over a finite state space $S$, let
\begin{enumerate}
\item The matrix $Q$ denote the transition probabilities over the non-absorbing
states
\item Let $\vec{\mu}_{0}$ be an arbitrary probability vector over the non-absorbing
states
\item Let $T_{0}\geq1$.
\end{enumerate}
Then there exists a unique quasi-stationary distribution $\vec{\mu}$
satisfying the equations
\begin{eqnarray*}
\vec{\mu}^{\tp}Q & = & \lambda\vec{\mu}^{\tp}\\
\vec{\mu}^{\tp}\vec{1} & = & 1\\
\vec{\mu} & \geq & 0
\end{eqnarray*}
and Algorithm \ref{enu:mainalgo} converges to the point $(\vec{\mu},\,\frac{1}{1-\lambda})$
with probability 1.

Furthermore, if $\lambda_{PV}$ is the principal eigenvalue of $Q$
and $\lambda_{NPV}$ are the other eigenvalues and they satisfy
\[
Re\left(\frac{1}{1-\lambda_{NPV}}\right)<\frac{1}{2}\left(\frac{1}{1-\lambda_{PV}}\right)\quad\forall\lambda_{NPV}\;\textrm{non-principal eigenvalues}.
\]
Then
\[
\sqrt{n}(\vec{\mu}_{n}-\vec{\mu})\rightarrow^{d}N(0,V)
\]
for some covariance matrix $V$. \end{thm}
\begin{proof}
The full proof in the Appendix \ref{sub:Proof-of-DT} but we outline
the main idea here. The technique uses the ODE method (\cite{kushner2003stochastic}
Theorem 5.2.1) where we are required to examine the asymptotic behavior
of the coupled dynamical system below. Here we neglect the asymptotically
negligible dependence on $n$ in order to illustrate the main idea.
The dynamical system of interest is
\begin{eqnarray*}
\dot{\vec{\mu}}(t) & = & \mathbb{E}_{\vec{\mu}(t),T(t)}\left[\frac{\sum_{l=0}^{\tau-1}\left(\mathbb{I}(X_{l}=\cdot|X_{0})\right)-\tau\vec{\mu}(t)}{T(t)}\right]\\
 & = & \frac{1}{T}\left[\vec{\mu}(t)^{\tp}A-(\vec{\mu}(t)^{\tp}A\vec{1})\vec{\mu}^{\tp}(t)\right]\quad\mbox{where A\ensuremath{\triangleq}(I-Q\ensuremath{)^{-1}}}
\end{eqnarray*}

\begin{eqnarray*}
\dot{T}(t) & = & \mathbb{E}_{\vec{\mu}(t)}[\tau]-T(t)\\
 & = & \vec{\mu}(t)^{\tp}(I-Q)^{-1}\vec{1}-T(t)
\end{eqnarray*}
where $\vec{\mu}(t)\in\mathbb{R}^{m}$ and $T(t)\in\mathbb{R}^{+}$.
($m$ is the number of non-absorbing states of the Markov chain)

Again in the Appendix \ref{sub:Proof-of-DT}, we are able to show
that for a given initial position in the probability simplex, the
solution to the above dynamical system exists and converges to its
stationary point which is the unique point that satisfies
\begin{eqnarray*}
\vec{\mu}^{\tp}Q & = & \rho\vec{\mu}^{\tp}\\
\sum\mu_{i} & = & 1\\
\mu_{i} & \geq & 0
\end{eqnarray*}
and $\rho=1-\frac{1}{E_{\vec{\mu}}(\tau)}$.

By Theorem 5.2.1 from \cite{kushner2003stochastic}, we can conclude
that $\vec{\mu}_{n}$ converges to the quasi-stationary distribution
for all initial configurations $(\vec{\mu}_{0},T_{0})$. 

Equation (\ref{eq:mainalg_formal}) can be analyzed for its rate of
convergence. Here we invoke the Theorem 10.2.1 of \cite{kushner2003stochastic}.
Because our algorithm uses a step size of $O(\frac{1}{n})$, we eventually
conclude that a CLT exists as long as the Jacobian matrix of the ODE
vector field has spectral radius less than $-\frac{1}{2}$. This is
equivalent to requiring that
\begin{equation}
Re\left(\frac{1}{1-\lambda_{NPV}}\right)<\frac{1}{2}\left(\frac{1}{1-\lambda_{PV}}\right)\quad\forall\lambda_{NPV}\;\textrm{non-principal eigenvalues}\label{eq:suffcond}
\end{equation}
where the $\lambda$'s are the eigenvalues of the $Q$ matrix.
\end{proof}

\section{Variations on the Existing Algorithm with Improved Rate of Convergence\label{sec:Variations}}

One interesting question to ask is what happens when the sufficient
conditions for CLT are not met. We will study a simple example consisting
of two states.

\subsection{Counter Example to CLT\label{sub:Counter-example}}

Imagine we have a Markov chain with three states $\{0,1,2\}$ and
transition matrix
\[
\left[\begin{array}{ccc}
1 & 0 & 0\\
\epsilon & \frac{1-\epsilon}{2} & \frac{1-\epsilon}{2}\\
\epsilon & \frac{1-\epsilon}{2} & \frac{1-\epsilon}{2}
\end{array}\right].
\]
Obviously the state $\{0\}$ is the absorbing state. In this setup,
because of symmetry, our Algorithm \ref{enu:mainalgo} reduces to
\begin{enumerate}
\item With probability $\frac{1-\epsilon}{2}$ sample either the state 1
or 2 (without knowing the previous state. This is OK by symmetry)
and add to the empirical distribution.
\item With probability $\epsilon$, sample from either 1 or 2 according
to the empirical distribution up until this point.
\end{enumerate}
We recognize this as a self-interacting Markov chain.

A self-interacting Markov chain (SIMC) \cite{delMoralSIMC} is a stochastic
process $\{X_{n}\}$ such that
\[
\mathbb{P}\left(X_{n+1}\in dx|\mathcal{F}_{n}\right)=\Phi(S_{n})(dx)
\]
where $\Phi$ is a function that transforms one measure into another
measure and $S_{n}$ is the empirical measure generated by $\{X_{k}\}_{k=0}^{n}$.

Our Algorithm \ref{enu:mainalgo} for the above ``loopy Markov chain''
reduces to the empirical process of a SIMC $X_{n}$ governed by the
functional
\[
\mathbb{P}(X_{n+1}=dz|\mathcal{F}{}_{n})=\int K(x,dz)dS_{n}(dx)
\]
where the kernel is given by
\[
K(x,dz)=\epsilon\delta_{x}(dz)+\left(\frac{1-\epsilon}{2}\right)\left[\delta_{1}(dz)+\delta_{2}(dz)\right]
\]

The sufficient condition for CLT (Equation (\ref{eq:suffcond})) in
this case translates to requiring $\epsilon<0.5$. 

When the CLT is violated however, \cite{delMoralSIMC} states that
over a very general class of bounded and measurable functions $f$
\[
\mathbb{E}[(S_{n}(f)-\bar{S}_{n}(f))^{2}]=\Theta\left(\frac{1}{n^{2(1-\epsilon)}}\right)
\]
where $S_{n}(f)=\int f(x)dS_{n}(x)$, $\bar{S}_{n}(f)=\mathbb{E}[S_{n}(f)]$.
Although this doesn't technically contradict with the existence of
a $\sqrt{n}$-CLT, it does suggest that the scaling sequence is $n^{1-\epsilon}$
instead of $\sqrt{n}$.

In the numerical experiment (Section \ref{sec:NumericalExp.}), we
simulate this example and demonstrate the slow rate of convergence
when $\epsilon<0.5$.

\subsection{Projection Algorithm and Polyak-Ruppert Averaging\label{sub:Projection-Algorithm}}

\subsubsection*{Doeblinization and the need for strong CLT}

The expected time to absorption $\mathbb{E}[\tau]$ is $\frac{1}{1-\lambda}$
where $\lambda$ is the principle eigenvalue of the substochastic
matrix $Q$. If $\mathbb{E}[\tau]$ is large, then the iterations
of the algorithm will take prohibitively long. One trick that can
be used is to ``Doeblinize'' the chain.

If we multiply $Q$ by a constant $\alpha<1$, this does not change
the eigenvector but shrinks the all the eigenvalues by the same proportion.
That means we can force the iterations to jump to absorption very
quickly. However, because of the non-linearity of $\frac{1}{1-\lambda}$
and its presence in the sufficient condition of the CLT (Equation
\ref{eq:suffcond}), the CLT condition will fail to hold if $\alpha$
is too small. We need a technique where CLT can always be guaranteed
regardless of the eigenvalues of the matrix $Q$.

Remark: in continuous-time, we can subtract $\alpha I$ matrix from
the transition rate matrix to achieve Doeblinization.

\subsubsection*{Projection algorithm}

By putting our algorithm into the stochastic approximations framework,
we can modify the algorithm into the projection-variant.
\begin{equation}
\vec{\mu}_{n+1}=\Theta_{H}\left[\vec{\mu}_{n}+\epsilon_{n}\left(\sum_{k=0}^{\tau^{(n+1)}-1}\mathbb{I}(X_{k}^{(n+1)}=\cdot|X_{0}^{(n+1)}\sim\vec{\mu}_{n})-\vec{\mu}_{n})\right)\right]\label{eq:Projection-algorithm}
\end{equation}
where the $\Theta_{H}$ denotes a $L_{2}$-projection into the probability
simplex. Of course we still require $\sum\epsilon_{n}=\infty$ and
$\sum\epsilon_{n}^{2}<\infty$. Notice that in practice, we only need
to perform very few number of projections. The expression inside the
projection operator always sum to one. So projection is only needed
if any component inside $\Theta$ becomes negative. Breaking it down
allows us to gain insight into when it becomes negative
\[
\vec{\mu}_{n}(1-\epsilon_{n}\tau^{(n+1)})+\epsilon_{n}\left(\sum I(...)\right).
\]
This can only be negative if $\tau>\frac{1}{\epsilon_{n}}$. But $\epsilon_{n}\downarrow0$
means this won't happen very often. The advantage of this version
is that we are free to use slower step sizes that weakens the condition
required for CLT to hold. Specifically, when $\epsilon_{n}=\Theta(\frac{1}{n^{\alpha}})$
for $\alpha<0.5$, a $\frac{1}{\sqrt{\epsilon_{n}}}$-CLT always hold.
\begin{thm}
\label{thm:projection-thm}Given an irreducible absorbing Markov chain
over a finite state space $S$, let
\begin{enumerate}
\item The matrix $Q$ denote the transition probabilities over the non-absorbing
states
\item Let $\vec{\mu}_{0}$ (the initial $\vec{\mu}$) be a probability vector
over the non-absorbing states
\item Let $T_{0}\geq1$.
\end{enumerate}
Then there exists a unique quasi-stationary distribution $\vec{\mu}$
satisfying the equations

\begin{eqnarray*}
\vec{\mu}^{\tp}Q & = & \lambda\vec{\mu}^{\tp}\\
\vec{\mu}^{\tp}\vec{1} & = & 1\\
\vec{\mu} & \geq & 0
\end{eqnarray*}
and the projection algorithm (Equation \ref{eq:Projection-algorithm})
converges to the point $\vec{\mu}$ with probability 1.

If step sizes are such that $\epsilon_{n}=\Theta\left(\frac{1}{n}\right)$
and if $\lambda_{PV}$ is the principal eigenvalue of $Q$ and $\lambda_{NPV}$
are the other eigenvalues and they satisfy
\[
Re\left(\frac{1}{1-\lambda_{NPV}}\right)<\frac{1}{2}\left(\frac{1}{1-\lambda_{PV}}\right)\quad\forall\lambda_{NPV}\;\textrm{non-principal eigenvalues}.
\]
Furthermore, we can conclude 
\[
\sqrt{n}(\vec{\mu}_{n}-\vec{\mu})\rightarrow^{d}N(0,V)
\]
for some covariance matrix $V$.

In the case that the step sizes are such that $\epsilon_{n}=\Theta\left(\frac{1}{n^{\alpha}}\right)$
for $0.5<\alpha<1$, we can conclude that (regardless of the eigenvalues
of $Q$) 
\[
\sqrt{n^{\alpha}}(\vec{\mu}_{n}-\vec{\mu})\rightarrow^{d}N(0,V)
\]
for some covariance matrix $V$.\end{thm}
\begin{proof}
The proof for the case of step size $\epsilon_{n}=\Theta\left(\frac{1}{n}\right)$
is almost identical to what's given in Section \ref{sub:Proof-of-DT}
after omitting the extra dimension $T_{n}$. In the case of $\epsilon_{n}=\Theta\left(\frac{1}{n^{\alpha}}\right)$
for $\alpha<0.5$, under the notation of Theorem \ref{thm:hurwitz},
we need to ensure that $J$ is Hurwitz as opposed to the stronger
condition that $J+\frac{I}{2}$ is Hurwitz. This is equivalent to
the condition that, (again under the notation of Theorem \ref{thm:hurwitz})
\[
Re(\lambda_{B})<\beta
\]
which is trivially always true by the Perron-Frobenius theorem \cite{karlin_taylor}.
Hence we can conclude that 
\[
\sqrt{n^{\alpha}}(\vec{\mu}_{n}-\vec{\mu})\rightarrow^{d}N(0,V)
\]
 by invoking Theorem 10.2.1 of \cite{kushner2003stochastic}.
\end{proof}

\subsubsection*{Polyak-Ruppert Averaging}

The Polyak-Ruppert averaging technique \cite{polyak1992acceleration},
(Theorem 11.1.1 in \cite{kushner2003stochastic} can be applied to
the projection algorithm to ensure that $\sqrt{n}$-CLT always holds
as long as we pick the step sequence to be $\epsilon_{n}=\Theta(\frac{1}{n^{\alpha}})$
for $\alpha<0.5$.
\begin{thm}
\label{thm:polyak-thm}

Given an irreducible absorbing Markov chain over a finite state space
$S$, let
\begin{enumerate}
\item The matrix $Q$ denote the transition probabilities over the non-absorbing
states
\item Let $\vec{\mu}_{0}$ (the initial $\vec{\mu}$) be a probability vector
over the non-absorbing states
\item Let $T_{0}\geq1$.
\end{enumerate}
Then there exists a unique quasi-stationary distribution $\vec{\mu}$
satisfying the equations

\begin{eqnarray*}
\vec{\mu}^{\tp}Q & = & \lambda\vec{\mu}^{\tp}\\
\vec{\mu}^{\tp}\vec{1} & = & 1\\
\vec{\mu} & \geq & 0
\end{eqnarray*}
and the step sizes satisfy $\epsilon_{n}=\Theta\left(\frac{1}{n^{\alpha}}\right)$
for $\alpha<0.5$. We can conclude that the averaged sequence 
\[
\vec{\nu}_{n}=\frac{1}{n}\sum_{k=1}^{n}\vec{\mu}_{k}
\]
 converges to the point $\vec{\mu}$ with probability 1.

Furthermore, a strong CLT always hold 
\[
\sqrt{n}(\vec{\nu}_{n}-\vec{\mu})\rightarrow^{d}N(0,V)
\]
for some covariance matrix $V$.
\end{thm}

\section{Algorithm for Continuous-Time Markov Chains\label{sec:CTMC}}

\subsection{Formulation and Convergence}

So far, the exposition has assumed that the Markov chain of interest
is a discrete-time process. It is straightforward to adapt our method
for continuous-time processes (such as the contact process). If we
denote the transition rate matrix of the CTMC in the following block
form
\[
T=\left[\begin{array}{cc}
0 & 0\\
N & Q
\end{array}\right]
\]
then we can write the algorithm as

\begin{eqnarray}
\vec{\mu}_{n+1}(x) & = & \vec{\mu}_{n}(x)+\nonumber \\
 &  & \frac{1}{n+1}\frac{\int_{0}^{\tau^{n+1,m}}\left(\mathbb{I}(X^{n+1}(s)=x|X_{0}^{n+1}\sim\vec{\mu}_{n})\right)ds-\tau^{n+1}\vec{\mu}_{n}(x)}{\frac{1}{n+1}\sum_{l=0}^{n+1}\tau^{l}}\label{eq:ct-alg-formal}\\
T_{n+1} & = & T_{n}+\frac{1}{n+2}\left(\tau^{n+1}-T_{n}\right).\nonumber 
\end{eqnarray}

By a similar approach as the discrete-time case, we deduce the related
dynamical system
\[
\begin{cases}
\dot{\vec{\mu}}(t) & =-\frac{1}{T(t)}\left(\vec{\mu}(t)^{\tp}Q^{-1}-(\vec{\mu}(t)^{\tp}Q^{-1}\vec{1})\vec{\mu}(t)^{\tp}\right)\\
\dot{T}(t) & =-\vec{\mu}(t)^{\tp}Q^{-1}\vec{1}-T(t).
\end{cases}
\]
It is straightforward to adapt the Perron-Frobenius theorem to transition
rate matrices such as $Q$ by decomposing $Q=A-bI$ where $A$ is
an irreducible matrix. We know the existence of a principal eigenvector
of positive entries $\bar{\vec{\mu}}$ (with eigenvalue smaller than
$0$) such that
\[
\bar{\vec{\mu}}^{\tp}Q=\bar{\lambda}\bar{\vec{\mu}}^{\tp}.
\]
The rest of the proof is very similar to the discrete-time case. The
only trick is to show that $\exp(-Q^{-1})$ is a matrix of non-negative
entries. That is included in the Lemma \ref{lem:nonnegative-entries}.
We summarize it in theorem form
\begin{thm}
\label{thm:CTMC-thm}Given an irreducible absorbing Markov chain over
a finite state space $S$, let
\begin{enumerate}
\item The matrix $Q$ denote the transition rates over the non-absorbing
states
\item Let $\vec{\mu}_{0}$ (the initial $\vec{\mu}$) be a probability vector
over the non-absorbing states
\item Let $T_{0}\geq1$.
\end{enumerate}
Then there exists a unique quasi-stationary distribution $\vec{\mu}$
satisfying the equations

\begin{eqnarray*}
\vec{\mu}^{\tp}Q & = & \lambda\vec{\mu}^{\tp}\\
\vec{\mu}^{\tp}\vec{1} & = & 1\\
\vec{\mu} & \geq & 0
\end{eqnarray*}
and the continuous-time algorithm (Equation \ref{eq:ct-alg-formal})
converges to the point $(\vec{\mu},\,-\frac{1}{\lambda})$ with probability
1.
\end{thm}

\subsection{Rate of Convergence}

In the notation of the definition of Equation \ref{eq:full_ODE}.
The Jacobian of the dynamical system is given by
\begin{eqnarray*}
\nabla_{\vec{\mu}}\bar{\vec{f}} & = & -\frac{1}{T}\left(Q^{-1}-Q^{-1}\vec{1}\vec{\mu}^{\tp}-(\vec{\mu}^{\tp}Q^{-1}\vec{1})I\right)\\
\nabla_{T}\bar{\vec{f}} & = & \frac{1}{T^{2}}\left(\vec{\mu}^{\tp}Q^{-1}-(\vec{\mu}^{\tp}Q^{-1}\vec{1})\vec{\mu}^{\tp}\right)\\
\nabla_{\vec{\mu}}\bar{h} & = & -Q^{-1}\vec{1}\\
\nabla_{T}\bar{h} & = & -1.
\end{eqnarray*}
When evaluated at the stationary point $(\bar{\vec{\mu}},\bar{T})$,
we get the matrix
\[
\left[\begin{array}{cc}
-\bar{\lambda}\left(Q^{-1}-Q^{-1}\vec{1}\bar{\vec{\mu}}^{\tp}-\frac{1}{\bar{\lambda}}I\right) & -Q^{-1}\vec{1}\\
\vec{0} & -1
\end{array}\right].
\]
Using similar techniques as the discrete-time case (given in Appendix
\ref{sub:Proof-of-DT}), we conclude that if $\lambda_{NPV}$ is any
non-principal eigenvalue of $Q$, then the sufficient condition for
CLT becomes
\[
2\lambda_{PV}>Re(\lambda_{NPV}).
\]

\begin{thm}
For the continuous-time algorithm (Equation \ref{eq:ct-alg-formal}),
if $\lambda_{PV}$ is the principal eigenvalue of $Q$ and $\lambda_{NPV}$
are the other eigenvalues and they satisfy
\[
2\lambda_{PV}>Re(\lambda_{NPV}).
\]
Furthermore, we can conclude 
\[
\sqrt{n}(\vec{\mu}_{n}-\vec{\mu})\rightarrow^{d}N(0,V)
\]
for some covariance matrix $V$.
\end{thm}
We can easily convert Equation \ref{eq:ct-alg-formal} to the projected
version and similar theorems regarding projection and Polyak-averaging
(Theorems \ref{thm:projection-thm} and \ref{thm:polyak-thm}) hold.

\subsection{Uniformization}

Because these CTMC have finite state space, we can form the associated
uniformized Markov chain. Let $Q$ be the transition rate matrix of
the non-absorbing states and let $\nu=\max_{i}(-q_{ii})$, we can
form a discrete-time transition matrix
\[
\tilde{Q}=I+\frac{1}{\nu}Q.
\]
It is straightforward to verify that any principal left-eigenvector
to $Q$ is also a principal left-eigenvector to $\tilde{Q}$. Hence
we apply the discrete-time algorithm to this DTMC.

\section{Numerical Experiments\label{sec:NumericalExp.}}

\subsection{Loopy Markov Chain}

Let's consider the loopy Markov chain given by the full transition
probability matrix
\[
\left[\begin{array}{ccc}
1 & 0 & 0\\
\epsilon & \frac{1-\epsilon}{2} & \frac{1-\epsilon}{2}\\
\epsilon & \frac{1-\epsilon}{2} & \frac{1-\epsilon}{2}
\end{array}\right].
\]
The eigenvalues of the sub-stochastic matrix are $1-\epsilon$ and
$0$. Hence the sufficient condition for CLT to hold is to require
$\epsilon<0.5$. We tested the original algorithm and the Polyak averaging
algorithm for the case of $\epsilon=0.98$, well outside of the CLT
sufficient condition. The result can be seen in Figure \ref{fig:loopyMC}
where the improved Polyak averaging algorithm significantly outperforms
the vanilla algorithm.

\subsection{M/M/1 queue with finite capacity and absorption}

We also simulated a M/M/1 queue where the system has a queue capacity
as well as an absorbing state when the system is empty. A discrete-time
Markov chain is created when we considered the arrival times of new
customers. The system we have simulated has a capacity of $100$ with
$\rho=1.25$. The expected time to absorption $E(\tau)$ is very large
so we Doeblinized the Markov chain by multiplying the probability
matrix by $0.95$. The Doeblinized Markov chain no longer satisfies
the CLT. You can see in Figure \ref{fig:MM1} that the Polyak averaging
algorithm significantly outperforms the vanilla algorithm.

\subsection{Contact Process on Complete graph\label{sub:Contact-Process}}

We now introduce the contact process. It's a class of models that
fall within the interacting particle systems framework whose quasi-stationary
distribution are important to physicists \cite{dickman,latestoliveira,marro,oliveiradickman1}.
\begin{defn}
A contact process is a continuous-time Markov chain (CTMC)$(X_{1}^{t},...,X_{n}^{t})\in\{0,1\}^{n}$,
where $t\geq0$ is the time, with an associated connected graph $(V,E)$
such that
\begin{itemize}
\item $|V|=n$.
\item Individual nodes transition from $1$ to $0$ at an exponential rate
of $1$.
\item Individual nodes transition from $0$ to $1$ at rate $\lambda r$
where $r$ is the fraction of neighbors that are in state $1$.
\end{itemize}
\end{defn}
This CTMC has $2^{n}$ states. The state $(0,0,\ldots,0)$ is an absorbing
state and the remaining states are all transient.

This CTMC will eventually reach the absorbing state but physicists
are interested in the ``pseudo-equilibrium'' behavior in the period
before absorption happens \cite{dickman,latestoliveira,marro,oliveiradickman1}.
In another words, we need an algorithm for estimating the quasi-stationary
distribution of this process. The difficulty is that the state space
is exponential in size save for a few special cases.

Here we simulate the contact process on a complete graph. If the infection
rate is changed to $1.5$, then each iteration of the algorithm would
take an extreme long time. We applied the version of the algorithm
designed for continuous-time Markov chains and Doeblinized the Markov
chain by subtracting $0.5I$ from the transition rate matrix. The
eigenvalue condition fails resulting in a slow rate of convergence
for the vanilla algorithm. The Polyak's averaging algorithm significantly
outperforms the vanilla algorithm. See Figure \ref{fig:cpcg}.

\section{Discussion and Conclusion}

In summary, we have improved upon the algorithm of \cite{oliveiradickman1}
by recognizing it as a stochastic approximation algorithm as opposed
to an urn process. In doing so, we were able to prove its law of large
number and CLT. The result is stronger than the results given in the
urn process literature\cite{athreya}. Furthermore, we provided a
counterexample that strongly suggests that the sufficient eigenvalues
condition for the CLT is also necessary and fails in many common applications.
An improved algorithm that uses projection and iterate averaging significantly
improves rate of convergence.

We have tested our algorithm on countable state space processes such
as the M/M/1/$\infty$ queue with success. Proving the convergence
of this algorithm in this countable state space setting is currently
an open problem. We're also working on a version of the algorithm
for estimating the quasi-stationary distribution of diffusion processes
using stochastic approximation.

Another open issue is how to pick the best Doeblinization constant.
When $\mathbb{E}[\tau]$ is large you're more likely to satisfy the
condition for the CLT but that's when run-time of the algorithm increases
proportionally. There must be a balance between the run-time of each
tour and the rate of convergence of $\vec{\mu}_{n}$. It is also not
clear what the optimal step size should be for the projected algorithm.

Finally, it would be very interesting to investigate the connection
between the phase transition critical point of contact processes and
its CLT critical point. Unfortunately, preliminary work seems to suggest
that those two are unconnected.

\section*{Acknowledgments}

Support from the NSF foundation through the grants CMMI-0846816 and
CMMI-1069064 is gratefully acknowledged.

\section*{Proof of Main Results}

\subsection{Proof of Discrete-Time Theorem \ref{thm:DTMC-unprojected}\label{sub:Proof-of-DT}}

We first restate a series of assumptions \& notations that is used
by Theorem 5.2.1 from \cite{kushner2003stochastic} which we will
invoke. Again the form of the recursion is $\vec{\theta}_{n+1}=\vec{\theta}_{n}+\epsilon_{n}\vec{W}_{n}$
where $\vec{W}_{n}$ is a martingale difference sequence with respect
to the filtration $\mathscr{F}_{n}$ that at least contains $\sigma(\vec{\theta}_{i},\vec{W}{}_{i-1},i\leq n)$.
Recall that for us, $\vec{W}_{n}$ consists of the two components
$\vec{W}_{n}$ (the probability vector) and $Z_{n}$ (and added time
dimension) defined in Equation \ref{eq:iterates_defined} and $\vec{\theta}_{n}$
consists of $\vec{\mu}_{n}$ and $T_{n}$.
\begin{enumerate}
\item $\epsilon_{n}\downarrow0,\;\sum\epsilon_{n}=\infty,\;\sum\epsilon_{n}^{2}<\infty$.
This is trivially satisfied for our $\epsilon_{n}=\frac{1}{n}$.
\item The observed responses have to have uniformly bounded variance: $\sup_{n}\mathbb{E}|\vec{W}_{n}|^{2}<\infty$.
See Lemma \ref{lem:uniform-boundedness}.
\item (A local-averaging condition) \label{enu:defn-of-g}Let $\vec{g}_{n}(\vec{\mu}_{n},T_{n})\triangleq\mathbb{E}[\vec{W}_{n}|\mathscr{F}_{n}]$.
The functions $\vec{g}_{n}(\vec{\mu},T)$ need to be continuous uniformly
in $n$, and there needs to exist a continuous function $\bar{\vec{g}}(\vec{\mu},T)$
such that for each $(\vec{\mu},T)$
\[
\lim_{n\rightarrow\infty}\left|\sum_{i=n}^{m(t_{n}+t)}\epsilon_{i}[\vec{g}_{i}(\vec{\mu},T)-\bar{\vec{g}}(\vec{\mu},T)]\right|\rightarrow0
\]
for each $t>0$. For the proof see Lemma \ref{lem:local-avging}
\end{enumerate}
Under these assumptions, Theorem 5.2.1 of \cite{kushner2003stochastic}
tells us that if the ODE $\frac{d}{dt}(\vec{\mu}(t),T(t))=\bar{\vec{g}}(\vec{\mu}(t),T(t))$
has an attractor (asymptotically stable point) with domain $A$ and
the sequence $(\vec{\mu}_{n},T_{n})$ visits a compact subset within
the domain infinitely often with probability $1$, then $(\vec{\mu}_{n},T_{n})$
converges to the attractor with probability $1$.

In our situation, it turns out that the entirely probability simplex
is the domain for an attractor situated at the quasi-stationary vector.
We will first compute the functions $\vec{g}_{n}$ and verify condition
3, then the uniformly bounded variance condition 2, and finally the
asymptotic behavior of the associated ODE.

\subsubsection{Local-averaging of the gradient field}
\begin{lem}
\label{lem:local-avging}Given the gradient field $\bar{\vec{g}}=\left(\begin{array}{c}
\bar{\vec{f}}\\
\bar{h}
\end{array}\right)$ defined by components

\begin{eqnarray*}
\bar{\vec{f}}(\vec{\mu},T) & \triangleq & \mathbb{E}_{\vec{\mu},T}\left[\frac{\sum_{l=0}^{\tau-1}\left(I(X_{l}=\cdot)-\vec{\mu}\right)}{T}\right]\\
\bar{h}(\vec{\mu},T) & = & \mathbb{E}_{\vec{\mu},T}[\tau-T]
\end{eqnarray*}
corresponding respectively to the dynamics of $\vec{Y}_{n}$ and $Z_{n}$,
we have for $\vec{g}_{n}(\vec{\mu}_{n},T_{n})=\mathbb{E}[\vec{W}_{n}|\mathscr{F}_{n}]$
\[
\lim_{n\rightarrow\infty}\left|\sum_{i=n}^{m(t_{n}+t)}\epsilon_{i}[\vec{g}_{i}(\vec{\mu},T)-\bar{\vec{g}}(\vec{\mu},T)]\right|\rightarrow0
\]
for each $t>0$ pointwise. Furthermore, $\vec{g}_{n}$ are continuous
uniformly in $n$.\end{lem}
\begin{proof}
We treat the $\vec{\mu}_{n}$ components and the $T_{n}$ component
of Equation \ref{eq:mainalg_formal} separately. Let us first define
and compute $\vec{f}_{n}(\vec{\mu}_{n},T_{n})\triangleq\mathbb{E}_{\vec{\mu}_{n},T_{n}}[\vec{Y}_{n}|\mathscr{F}_{n}]$. 

It is clear that $\vec{Y}_{n}(x,\vec{\mu}_{n},T_{n})\rightarrow_{n\rightarrow\infty}\frac{\sum_{l=0}^{\tau^{(n+1)}-1}\left(\mathbb{I}(X_{l}^{n+1}=x)-\vec{\mu}_{n}(x)\right)}{T_{n}}$
where $x$ is a component of the vector $\vec{Y}_{n}$ and $\vec{\mu}_{n}$,
$T_{n}$ are fixed arguments. We can apply the dominated convergence
theorem to arrive at the conclusion
\[
\vec{f}_{n}(\vec{\mu},T)=\mathbb{E}[\vec{Y}_{n}|\mathscr{F}_{n},\vec{\mu}_{n}=\vec{\mu},T_{n}=T]\rightarrow_{n\rightarrow\infty}\mathbb{E}_{\vec{\mu},T}\left[\frac{\sum_{l=0}^{\tau-1}\left(\mathbb{I}(X_{l}=\cdot)-\vec{\mu}\right)}{T}\right].
\]
Let's define the limit to be $\bar{\vec{f}}(\vec{\mu},T)$
\[
\bar{\vec{f}}(\vec{\mu},T)\triangleq\mathbb{E}_{\vec{\mu},T}\left[\frac{\sum_{l=0}^{\tau-1}\left(\mathbb{I}(X_{l}=\cdot)-\vec{\mu}\right)}{T}\right].
\]

We now have
\begin{align*}
\lim_{n\rightarrow\infty}\left|\sum_{i=n}^{m(t_{n}+t)}\epsilon_{i}\left(\vec{f}_{i}(\vec{\mu},T)-\bar{\vec{f}}(\vec{\mu},T)\right)\right| & \leq\lim_{n\rightarrow\infty}\sum_{i=n}^{m(t_{n}+t)}\left|\epsilon_{i}(\vec{f}_{i}(\vec{\mu},T)-\bar{\vec{f}}(\vec{\mu},T))\right|\\
 & \leq\lim_{n\rightarrow\infty}t\max_{n\leq i\leq m(t_{n}+t)}|\vec{f}_{i}(\vec{\mu},T)-\bar{\vec{f}}(\vec{\mu},T)|\rightarrow0.
\end{align*}

For the $T_{n}$ component, define $\mathbb{E}[Z_{n}|T_{n},\vec{\mu}_{n}]=\mathbb{E}_{\vec{\mu}_{n}}[\tau-T_{n}]\triangleq\bar{h}(\vec{\mu}_{n,}T_{n})$.
This field is independent of $n$, hence it trivially satisfies the
above ``local averaging'' condition.

If we look at the expansion in Lemma \ref{lem:uniformness-of-error},
it is clear that the Jacobian $D\vec{f}_{n}$ would be uniformly bounded
in $n$ for local neighborhoods around each point $(\vec{\mu},T)$,
hence $\vec{f}_{n}$ would be continuous uniformly in $n$. 
\end{proof}

\subsubsection{Uniformly bounded variance}
\begin{lem}
\label{lem:uniform-boundedness}$\sup_{n}\mathbb{E}|\vec{W}_{n}|^{2}<\infty$
for the unprojected algorithm
\end{lem}
\begin{eqnarray*}
\mathbb{E}[|\vec{Y}_{n}|^{2}|\vec{\mu}_{n},T_{n}] & \leq & \mathbb{E}_{\vec{\mu}_{n},T_{n}}\left[\left|\sum_{l=0}^{\tau-1}(\mathbb{I}(X_{l}=\cdot)-\vec{\mu}_{n})\right|^{2}\right]\\
 & \leq & \mathbb{E}_{\vec{\mu}_{n}}(\tau^{2})\\
 & \leq & \sum_{n\geq0}\mathbb{P}_{\vec{\mu}_{n}}(\tau>\sqrt{n})\\
 & \leq & \vec{\mu}_{n}^{\tp}\left(\sum_{l=0}^{\infty}Q^{\left\lfloor \sqrt{l}\right\rfloor }\right)\vec{1}.
\end{eqnarray*}

\begin{proof}
Now, $\mathbb{E}[|\vec{Y}_{n}|^{2}]=\mathbb{E}\left[\mathbb{E}[|\vec{Y}_{n}|^{2}|\vec{\mu}_{n},T_{n}]\right]=\mathbb{E}(\vec{\mu}_{n})^{\tp}\left(\sum_{l=0}^{\infty}Q^{\sqrt{l}}\right)\vec{1}$.
The infinite sum can be shown to be convergence by an integral test.
Since $\mathbb{E}(\vec{\mu}_{n})$ is a vector in the probability
simplex, which is compact, it is bounded from above.

For the second $T_{n}$ component, we have $\mathbb{E}[Z_{n}^{2}|\vec{\mu}_{n},T_{n}]=\mathbb{E}_{\vec{\mu}_{n},T_{n}}(\tau-T_{n})^{2}\leq\mathbb{E}_{\vec{\mu}_{n}}(\tau^{2})$
because $T_{n}$ is non-negative. Following the argument above, this
is also bounded in $n$.
\end{proof}

\subsubsection{The dynamical system}

Lemma \ref{lem:local-avging} show that the dynamical system of interest
has gradient field $\overline{\vec{g}}$ consisting of
\begin{eqnarray}
\dot{\vec{f}}(\vec{\mu},T) & \triangleq & \mathbb{E}_{\vec{\mu},T}\left[\frac{\sum_{l=0}^{\tau-1}\left(\mathbb{I}(X_{l}=\cdot)-\vec{\mu}\right)}{T}\right]\label{eq:full_ODE}\\
\dot{\vec{h}}(\vec{\mu},T) & = & \mathbb{E}_{\vec{\mu},T}[\tau-T].\nonumber 
\end{eqnarray}
After some expansion, they become (writing $\vec{\mu}$ as a vector
ODE)
\begin{align*}
\dot{\vec{\mu}}(t)^{\tp} & =\frac{1}{T}\left[(\vec{\mu}(t)^{\tp}(I-Q)^{-1}-(\vec{\mu}(t)^{\tp}(I-Q)^{-1}\vec{1})\vec{\mu}^{\tp}(t)\right]\\
\dot{T}(t) & =\vec{\mu}(t)^{\tp}(I-Q)^{-1}\vec{1}-T(t).
\end{align*}

In the proof of Theorem 5.2.1 of \cite{kushner2003stochastic}, the
subsequence limit $\theta(\cdot,\omega)$ is a solution to the above
ODE. We only need to prove that these solutions converge to the quasi-stationary
distribution. The strategy is to prove the asymptotic limit of all
solutions of a reduced ODE starting in $H$ is the quasi-stationary
distribution, and then show that these particular solutions (subsequence
limits of $\theta(\cdot,\omega)$) of the full ODE (Equation \ref{eq:full_ODE})
can be converted into solutions for the reduced ODE. Finally we combine
these and show that these subsequence solutions that the iterates
$(\mu_{n},T_{n})$ visits a compact subset of $H\times(0,\infty)$
infinitely often almost surely.

The reduced ODE is
\begin{align}
\dot{\vec{\nu}}(t) & =\vec{\nu}(t)^{\tp}(I-Q)^{-1}-(\vec{\nu}(t)^{\tp}(I-Q)^{-1}\vec{1})\vec{\nu}(t)^{\tp}\label{eq:reduced_ODE}\\
\vec{\nu}(0) & =\vec{\mu}_{0}.\nonumber 
\end{align}

For convenience, first define
\begin{defn}
$\Gamma(t)=\int_{0}^{t}\frac{1}{T(s)}ds$

It is not hard to see that $\vec{\mu}(\Gamma^{-1}(t))$ is a solution
to the reduced ODE. The following Lemma ensures that the inverse is
well defined.\end{defn}
\begin{lem}
$\Gamma(t)$ is non-negative, increasing, and goes to $\infty$\label{lem:Gamma-lemma}\end{lem}
\begin{proof}
The increasing part is trivial because $T(s)$ is strictly positive
(in both the discrete-time and continuous-time cases). Let's assume
that $\int_{0}^{\infty}\frac{1}{T(s)}<\infty$. This implies 
\begin{eqnarray*}
T(t) & = & T_{0}\exp\left(\int_{0}^{t}\left[\frac{\mathbb{E}_{\vec{\mu}(s)}(\tau)}{T(s)}\right]ds-t\right)\\
 & \leq & T_{0}\exp\left(\sup_{\vec{\mu}\in H}\mathbb{E}_{\vec{\mu}}(\tau)\int_{0}^{t}\frac{1}{T(s)}ds-t\right)\\
 & < & T_{0}\exp(K\int_{0}^{t}\frac{1}{T(s)}ds)\\
 & < & \tilde{K}\quad\forall t.
\end{eqnarray*}
This means $\frac{1}{T(t)}\geq\tilde{K}\Rightarrow\int_{0}^{\infty}\frac{1}{T(s)}=\infty$
a contradiction.
\end{proof}
$\Gamma^{-1}(0)=0$ so $\vec{\mu}(\Gamma^{-1}(0))=\vec{\mu}(0)\in H$.
Now let's analyze the asymptotic behavior of any such solution $\vec{\nu}(t)$
to Equation \ref{eq:reduced_ODE}.

\begin{comment}
\begin{lem}
The only stationary point to the reduced ODE \ref{eq:reduced_ODE}
in the simplex is the quasi-stationary distribution $\bar{\vec{\mu}}$.\label{lem:unique_stationary}\end{lem}
\begin{proof}
If we search for the zero point of the gradient field we need to solve
the algebraic equation over the simplex
\[
\vec{\nu}^{\tp}(I-Q)^{-1}=(\vec{\nu}(t)^{\tp}(I-Q)^{-1}\vec{1})\vec{\nu}(t)^{\tp}.
\]
Notice that the quasi-stationary distribution can solve the above
equation. Notice we make no assumptions on the periodicity of $Q$
but nevertheless $(I-Q)^{-1}$ is always aperiodic. Combined with
the fact that $(I-Q)^{-1}$ is also an irreducible non-negative matrix
we may apply Perron-Frobenius theory to $(I-Q)^{-1}.$ Note $\vec{\nu}(t)^{\tp}(I-Q)^{-1}\vec{1}>0$
but Perron-Frobenius theorem says that all the non-negative eigenvector
(with some positive components) solutions to the equation lies on
a one-dimensional subspace. We can take the Perron-Frobenius eigenvector
and plug it into the above equation to verify that the right hand
is satisfied. Hence on the simplex, we can have only one stationary
point.\end{proof}
\end{comment}

\begin{lem}
Given any solution to the reduced ODE \ref{eq:reduced_ODE} such that
$\vec{\nu}(0)\in H$, they converge to the quasi-stationary distribution
$\bar{\vec{\mu}}$.\label{lem:convergence}\end{lem}
\begin{proof}
If $\vec{v}(0)\in H$, then the entire trajectory stays in $H$. Define
$A\triangleq(I-Q)^{-1}$ . By the Duhamel's principal, all solutions
to $\vec{\nu}$ can be represented by 
\[
\vec{\nu}(t)^{\tp}=\vec{v}(0)^{\tp}\exp\left(At-\int_{0}^{t}(\vec{\nu}(s)^{\tp}A\vec{1})ds\right).
\]
Because $A=\sum_{n=0}^{\infty}(Q)^{n}$ is a matrix with only non-negative
entries and $\vec{\nu}(0)\geq\vec{0}$, we have $\vec{\nu}(t)=\vec{v}(0)^{\tp}\exp(At)\exp\left(-\int_{0}^{t}(\vec{\nu}(s)A\vec{1}^{\tp})ds\right)\geq0$.
Along with $\vec{1}^{\tp}\vec{\nu}(0)=1$, we know that $\vec{\nu}(t)$
belongs to the simplex. The gradient field is continuously differntiable
over the simplex which is compact. Hence there exists unique solutions
to the ODE.

Rearranging the equation gives

\begin{eqnarray}
\vec{\nu}(t)^{\tp}\exp\left(\int_{0}^{t}(\vec{\nu}(s)^{\tp}A\vec{1})ds\right) & = & \vec{\nu}(0)^{\tp}\exp(At)\\
\vec{\nu}(t)^{\tp}\exp\left(\int_{0}^{t}(\vec{\nu}(s)^{\tp}A\vec{1})ds-\beta t\right) & = & \vec{v}(0)^{\tp}\exp(At-\beta t).\label{eq:lemmaeq}
\end{eqnarray}

Here, $\beta$ denotes the Perron-Frobenius eigenvalue for $A\triangleq(I-Q)^{-1}$.
Notice that regardless of the periodicity assumption on $Q$, $(I-Q)^{-1}$
is a strictly positive matrix. Consequently $\frac{e^{A}}{e^{\beta}}$
is a strictly positive matrix with spectral radius $1$. By Perron-Frobenius
theorem (\cite{karlin_taylor} Appendix Theorem 2.1), we have that
for some $\vec{w}$
\[
\vec{v}(0)^{\tp}\exp(An-\beta n)\rightarrow_{t\rightarrow\infty}\vec{w}
\]
where $\vec{w}$ is a multiple of the Perron-Frobenius eigenvector
of the matrix $A$. Because $(A-\beta)/m$ is also a matrix with the
same eigenvector, the above convergence will also hold along sequences
$\frac{n}{m}$ for fixed $m$ as $n\rightarrow\infty$. The exponential
is an uniformly continuous function in this case, so the convergence
also holds along the real numbers as $t\rightarrow\infty$.

Now, take inner product of Equation \ref{eq:lemmaeq} with $\vec{1}$
to obtain
\[
\exp\left(\int_{0}^{t}(\vec{\nu}(s)^{\tp}A\vec{1})ds-\beta t\right)\rightarrow_{t\rightarrow\infty}\gamma\triangleq<\vec{w},\vec{1}>.
\]
We now rewrite the original representation in the following way
\begin{eqnarray*}
\vec{\nu}(t) & = & \vec{\nu}(0)^{\tp}\exp\left(At-\int_{0}^{t}(\vec{\nu}(s)^{\tp}A\vec{1})ds\right)\\
 & = & \vec{\nu}(0)^{\tp}\exp(At-\beta t)\exp(-(\int_{0}^{t}(\vec{\nu}(s)^{\tp}A\vec{1})ds-\beta t))\\
 & \rightarrow_{t\rightarrow\infty} & \frac{\vec{w}}{\gamma}.
\end{eqnarray*}

Now the limit is a normalized quasi-stationary vector. The last fact
that finishes the Lemma is that the Perron-Frobenius eigenvector of
$A$ and $Q$ are identical so $\vec{\nu}(t)$ converges to the quasi-stationary
distribution of $Q$.\end{proof}
\begin{thm}
Any solution (if exists) solving 
\begin{align*}
\dot{\vec{\mu}}^{\tp}(t) & =\frac{1}{T}\left[(\vec{\mu}(t)^{\tp}(I-Q)^{-1}-(\vec{\mu}(t)^{\tp}(I-Q)^{-1}\vec{1})\vec{\mu}(t)^{\tp}\right]\\
\dot{T}(t) & =\vec{\mu}(t)^{\tp}(I-Q)^{-1}\vec{1}-T(t)
\end{align*}
with initial conditions $\vec{\mu}_{0}\in H$ and $T_{0}>0$ converges
to the quasi-stationary distribution in $\vec{\mu}$ and $E_{\bar{\vec{\mu}}}(\tau)=\frac{1}{1-\lambda}$
in $T$ where $\lambda$ is the principal eigenvalue of $Q$. The
random iterates $(\vec{\mu}_{n},T_{n})$ visits a compact subset (might
depend on $\omega$) of this attractor space ($H\times(0,\infty)$)
almost always. This implies that $\vec{\mu}_{n}\rightarrow\bar{\vec{\mu}}$
and $T_{n}\rightarrow\frac{1}{1-\lambda}$ with probability one.\end{thm}
\begin{proof}
Here we chain together the above few lemmas. We find that $\vec{\mu}(\Gamma^{-1}(t))$
is a solution to the reduced ODE with initial condition $\vec{\mu}(0).$
Therefore Lemma \ref{lem:convergence} tells us that $\vec{\mu}(\Gamma^{-1}(t))\rightarrow\bar{\vec{\mu}}.$
Furthermore, Lemma \ref{lem:Gamma-lemma} implies $\Gamma^{-1}(t)\rightarrow\infty$.
Together, it means $\vec{\mu}(t)\rightarrow\bar{\vec{\mu}}$.

$T(t)$ can be solved using the formula
\[
T(t)=\frac{\int_{0}^{t}\mathbb{E}_{\vec{\mu}(s)}[\tau]e^{s}ds+T_{0}}{e^{t}}.
\]
Because $\mathbb{E}_{\vec{\mu}(s)}[\tau]\rightarrow\mathbb{E}_{\bar{\vec{\mu}}}[\tau]$,
we can use L'Hopital's rule and get
\[
\lim_{t\rightarrow\infty}T(t)=\lim_{t\rightarrow\infty}\frac{\mathbb{E}_{\vec{\mu}(t)}[\tau]e^{t}}{e^{t}}=\mathbb{E}_{\bar{\vec{\mu}}}[\tau]=\frac{1}{1-\lambda}.
\]

Now one might notice that $T_{n}$ does not lie in a bounded set.
This could potentially lead to problems when the proof of \cite{kushner2003stochastic}
Theorem 5.2.1 assumes that $\int_{0}^{t}\bar{\vec{f}}(\vec{\mu}^{n}(s),T^{n}(s))ds$
and $\int_{0}^{t}\bar{h}(\vec{\mu}^{n}(s),T^{n}(s))ds$ are equicontinuous
classes of functions (for almost every $\omega$).

However, $T_{n}=\frac{1}{n+1}\sum_{k=1}^{n}\tau^{(k)}$, and we can
show that $T_{n}$ is bounded almost surely by a finite random variable
by Lemma \ref{lem:-Tn-boundedness}. This means $T^{n}(s,\omega)$
lives on a compact set for each fixed $\omega$ in a set of full measure.

We now satisfy all the requirements of Theorem 4.2.1 of \cite{kushner2003stochastic}
and our iterates $(\vec{\mu}_{n},T_{n})$ remains within the domain
of attraction ($H\times(0,\infty)$) of ($\bar{\vec{\mu}},\,\frac{1}{1-\lambda}$)
infinitely often. Hence the stochastic approximation algorithm iterates
converge to that point with probability 1.\end{proof}
\begin{lem}
\label{lem:-Tn-boundedness}$T_{n}$ is almost surely bounded by a
finite random variable\end{lem}
\begin{proof}
Recall that $T_{n}=\frac{1}{n+1}\sum_{k=1}^{n}\tau^{(k)}$. We can
consider $\tilde{\tau}(x)$ which is the stopping time of a Markov
chain starting from state $x$ and define
\[
\tilde{\tau}\triangleq\max\{\tilde{\tau}(x)|\forall x\in S\}.
\]
As a consequence,
\[
\mathbb{P}(\tau^{(k)}>y)\leq\mathbb{P}(\tilde{\tau}>y).
\]
That means we can couple the random variable $\tau^{(k)}$ with a
sequence of i.i.d. random variables $\tilde{\tau}^{(k)}$. Each one
has finite expectations because
\begin{align*}
\mathbb{E}[\tilde{\tau}] & =\sum_{n\geq0}\mathbb{P}(\tilde{\tau}>n)\\
 & \leq\sum_{n\geq0}\sum_{x\in S}\mathbb{P}(\tilde{\tau}(x)>n)\\
 & =\sum_{x\in S}\mathbb{E}(\tilde{\tau}(x))<\infty.
\end{align*}
This means $\lim_{n}T_{n}\leq\lim_{n}\frac{1}{n+1}\sum_{k=1}^{n}\tilde{\tau}^{(k)}\rightarrow\mathbb{E}[\tilde{\tau}]$
. Hence $T_{n}$ is almost surely bounded by a finite random variable.
\end{proof}

\subsection{Rate of convergence proof}

In trying to obtain a rate of convergence result for main algorithm
in the form of Equation\ref{eq:mainalg_formal}, we invoke Theorem
10.2.1 of \cite{kushner2003stochastic}. There's a whole set of assumptions
that need to be checked. Here, recall that $\vec{\theta}_{n}$ contains
two components, $\vec{\mu}_{n}$ and $T_{n}$ (Equation \ref{eq:mainalg_formal})
and that the notation is $\mathbb{E}_{n}(\vec{Y}_{n})\triangleq\mathbb{E}[\vec{Y}_{n}|\mathscr{F}_{n}]$.
We list the sufficient conditions here.
\begin{enumerate}
\item $\{\vec{W}_{n}\mathbb{I}_{\{|\vec{\theta}_{n}-\bar{\vec{\theta}}|\leq\rho\}}\}$
has to be uniformly integrable where $\overline{\vec{\theta}}$ is
the w.p. 1 limit of $\vec{\theta}_{n}$. This is trivial because $\sup_{n}\mathbb{E}|\vec{W}_{n}|^{2}<\infty$
by Lemma \ref{lem:uniform-boundedness}.
\item $\bar{\vec{\theta}}$, the limit point of the ODE, is an isolated
stable point. Again, it's trivial.
\item $\mathbb{E}(\vec{W}_{n}|\mathscr{F}_{n})=\vec{g}_{n}(\vec{\theta}_{n})$
can be expanded as
\[
\vec{g}_{n}(\vec{\theta})=\vec{g}_{n}(\bar{\vec{\theta}})+(D\vec{g}_{n})(\bar{\vec{\theta}})(\vec{\theta}-\bar{\vec{\theta}})+o(|\vec{\theta}-\bar{\vec{\theta}}|)
\]
 where the error $o$ is uniform in $n$. This is not so trivial and
the proof is given below in Lemma \ref{lem:uniformness-of-error}.
\item We need the sequence $\left\{ \frac{\vec{\theta}_{n}-\bar{\vec{\theta}}}{\sqrt{\epsilon_{n}}}\right\} $
to be tight. See Lemma \ref{lem:tightness}.
\item $\lim_{n,m}\frac{1}{\sqrt{m}}\sum_{i=n}^{n+mt-1}\vec{g}_{i}(\bar{\vec{\theta}})=0$
uniformly for each small t-interval. See Lemma \ref{lem:-sqrtn-averaging}.
\item There exists a matrix $A$ such that $(A+I/2)$ is Hurwitz and
\[
\lim_{n,m}\frac{1}{m}\sum_{i=n}^{n+m-1}[(D\vec{g}_{i})(\bar{\vec{\theta}})-A]=0.
\]
Let $A=(D\bar{\vec{g}})(\bar{\vec{\theta}})$ then the above is true
because $D\vec{g}_{i}(\bar{\vec{\theta}})\rightarrow D\bar{\vec{g}}(\bar{\vec{\theta}})$
by Lemma \ref{lem:Jacobian-convergence}. Conditions for $A+I/2$
being Hurwitz is given in Theorem \ref{thm:hurwitz}.
\item Define $\delta\vec{M}_{n}=\vec{W}_{n}-\mathbb{E}_{n}(\vec{W}_{n})$.
There exists a $p>0$ such that
\[
\sup_{n}\mathbb{E}|\delta\vec{M}_{n}|^{2+p}<\infty
\]
and a non-negative definite matrix $\Sigma$ such that
\[
\mathbb{E}_{n}\delta\vec{M}_{n}\delta\vec{M}_{n}^{\tp}\rightarrow\Sigma.
\]
This is proven in Lemma \ref{lem:quadratic-variation}.
\end{enumerate}

\subsubsection{Uniformness of the error terms}
\begin{lem}
\label{lem:uniformness-of-error}$\mathbb{E}(\vec{W}_{n}|\mathscr{F}_{n})=\vec{g}_{n}(\vec{\theta}_{n})$
can be expanded as
\[
\vec{g}_{n}(\vec{\theta})=\vec{g}_{n}(\bar{\vec{\theta}})+(D\vec{g}_{n})(\bar{\vec{\theta}})(\vec{\theta}-\bar{\vec{\theta}})+o(\left|\vec{\theta}-\bar{\vec{\theta}}\right|)
\]
 where the error $o$ is uniform in $n$.\end{lem}
\begin{proof}
Refer to proof of \ref{lem:local-avging} for the components of $\vec{g}_{n}$.
Again, the $\overline{h}$ component causes no problem because it
is independent of $n$. By defining $v(x,s)=E[e^{-s\tau}|X_{0}=x]$,
$\vec{f}_{n}(\vec{\mu},T)$, the $\vec{\mu}$ component of $\vec{g}_{n}$,
can be expanded as
\begin{align*}
\vec{f}_{n}(\vec{\mu},T)(x) & =\mathbb{E}_{\vec{\mu},T}\left[\int_{0}^{\infty}e^{-(T+\frac{\tau}{n+1})u}\left(\sum_{k=0}^{\tau-1}(\mathbb{I}(X_{k}=x)-\vec{\mu}(x))\right)du\right]\\
 & =\int_{0}^{\infty}\mathbb{E}_{\vec{\mu},T}\left[e^{-(T+\frac{\tau}{n+1})u}\sum_{k=0}^{\infty}\left(\mathbb{I}(\tau>k,X_{k}=x)-\mathbb{I}(\tau>k)\vec{\mu}(x)\right)\right]du\\
 & =\int_{0}^{\infty}e^{-Tu}\sum_{k=0}^{\infty}\left(e^{-k\frac{u}{n+1}}E_{\vec{\mu}}\left[e^{-(\tau-k)\frac{u}{n+1}}\left(\mathbb{I}(\tau>k,X_{k}=x)-\mathbb{I}(\tau>k)\vec{\mu}(x)\right)\right]\right)du\\
 & =\int_{0}^{\infty}e^{-Tu}\sum_{k=0}^{\infty}\left(e^{-k\frac{u}{n+1}}E_{\vec{\mu}}\left[v(X_{k},\frac{u}{n+1})(\mathbb{I}(\tau>k,X_{k}=x)-\mathbb{I}(\tau>k)\vec{\mu}(x))\right]\right)du\\
 & =\int_{0}^{\infty}e^{-Tu}\sum_{k=0}^{\infty}\left(e^{-k\frac{u}{n+1}}\left[v(x,\frac{u}{n+1})\vec{\mu}^{\tp}Q^{k}\delta_{x}-(\vec{\mu}^{\tp}Q^{k}\vec{v}(\cdot,\frac{u}{n+1}))\vec{\mu}(x)\right]\right)du\\
 & =\int_{0}^{\infty}e^{-Tu}\left[v(x,\frac{u}{n+1})\vec{\mu}^{\tp}(I-e^{-\frac{u}{n+1}}Q)^{-1}\delta_{x}-\vec{\mu}^{\tp}(I-e^{-\frac{u}{n+1}}Q)^{-1}\vec{v}(\cdot,\frac{u}{n+1}))\vec{\mu}(x)\right]du.
\end{align*}
It is easy to see that integration and (partial) differentiation can
be interchanged in this case because of the integrand's smoothness
and integrability. $v(x,s)$ is bounded by $1$ and $0$ for $s\geq0$.
The integrand consists of a second-order expression in $\vec{\mu}$
multiplied by an exponential damping factor $e^{-Tu}$. It is now
clear that all the mixed second partial derivatives of $\vec{f}_{n}$
will be bounded uniformly in $n$ \emph{in a neighborhood} around
the stationary point $\bar{\vec{\theta}}=(\bar{\vec{\mu}},\bar{\vec{T}})$
since $\bar{T}>0$.
\end{proof}

\subsubsection{Tightness of the normalized iterates}

In order for the CLT to hold, the normalized iterates $\left\{ \frac{\vec{\theta}_{n}-\bar{\vec{\theta}}}{\sqrt{\epsilon_{n}}}\right\} $
has to be tight.
\begin{lem}
\label{lem:tightness}The normalized iterates $\left\{ \frac{\vec{\theta}_{n}-\bar{\vec{\theta}}}{\sqrt{\epsilon_{n}}}\right\} $
is tight.\end{lem}
\begin{proof}
Here make a slight modification to the proof of Theorem 10.4.1 in
\cite{kushner2003stochastic}. We let $A=(D\overline{\vec{g}})(\bar{\vec{\theta}})$.
For any positive definite matrix $C$, there exists a positive definite
solution $P$ to the equation
\[
A^{\tp}P+PA=-C.
\]
 We take this $P$ and for each $A_{n}=(D\vec{g}_{n})(\bar{\vec{\theta}})$,
we obtain a sequence of matrices $C_{n}$ via
\[
A_{n}^{\tp}P+PA_{n}=-C_{n}.
\]
Obviously $C_{n}\rightarrow C$ and because $C$ is strictly positive
definite, there exists a $\lambda>0$ such that $C_{n}\succeq\lambda P$
in the positive-definite sense.

We use the Lyapunov function $V(\vec{\theta})=(\vec{\theta}-\bar{\vec{\theta}})^{\tp}P(\vec{\theta}-\bar{\vec{\theta}})$,
however, we now have to deal with the gradient field $\vec{g}_{n}$
as opposed to $\vec{g}$ in the proof of \cite{kushner2003stochastic}.
We will control the changes in the Lyapunov function by expanding
$\vec{g}_{n}$ around the stationary point $\bar{\vec{\theta}}$

\begin{align*}
\mathbb{E}[V(\vec{\theta}_{n+1})|\mathscr{F}_{n}]-V(\vec{\theta}_{n}) & =2\epsilon_{n}(\vec{\theta}_{n}-\bar{\vec{\theta}})^{\tp}P\vec{g}_{n}(\vec{\theta}_{n})+\epsilon_{n}^{2}\mathbb{E}_{n}(\vec{W}_{n}^{\tp}P\vec{W}_{n})\\
 & =2\epsilon_{n}(\vec{\theta}_{n}-\bar{\vec{\theta}})^{\tp}P\vec{g}_{n}(\bar{\vec{\theta}})+2\epsilon_{n}(\vec{\theta}_{n}-\bar{\vec{\theta}})^{\tp}PA_{n}(\vec{\theta}_{n}-\bar{\vec{\theta}})\\
 & \qquad+2\epsilon_{n}(\vec{\theta}_{n}-\bar{\vec{\theta}})^{\tp}Po(|\vec{\theta}_{n}-\bar{\vec{\theta}}|)+O(\epsilon_{n}^{2})\\
 & =O\left(\epsilon_{n}\frac{1}{n}\right)-\epsilon_{n}(\vec{\theta}_{n}-\bar{\vec{\theta}})^{\tp}C_{n}(\vec{\theta}_{n}-\bar{\vec{\theta}})+2\epsilon_{n}(\vec{\theta}_{n}-\bar{\vec{\theta}})^{\tp}Po(|\vec{\theta}_{n}-\bar{\vec{\theta}}|)+O(\epsilon_{n}^{2})\\
 & \leq O(\epsilon_{n}^{2})-\epsilon_{n}\tilde{\lambda}V(\vec{\theta}_{n})
\end{align*}
where the several facts are used
\begin{enumerate}
\item $C_{n}\succeq\lambda P$ for large $n$ (the inequality is in the
positive-definite sense).
\item $(\vec{\theta}_{n}-\bar{\vec{\theta}})^{\tp}Po(|\vec{\theta}_{n}-\bar{\vec{\theta}}|)\leq\delta V(\vec{\theta}_{n})$
for small $\delta$ and all $n$ large enough by Cauchy-Schwartz.
\item The error term $o(|\vec{\theta}_{n}-\bar{\vec{\theta}}|)$ of the
Taylor series expansion is uniform for all $\vec{g}_{n}$ as proven
in Lemma \ref{lem:uniformness-of-error}.
\item $\vec{g}_{n}(\bar{\vec{\theta}})=O(\frac{1}{n})$. This point is proven
in Lemma \ref{lem:rate-of-gn} below.
\end{enumerate}
At this point, we can use the rest of the proof of Theorem 10.4.1
of \cite{kushner2003stochastic} to show that $\mathbb{E}[V(\vec{\theta}_{n+1})|\mathscr{F}_{n}]=O(\epsilon_{n})$
which trivially leads to tightness.\end{proof}
\begin{lem}
\label{lem:rate-of-gn}$\vec{g}_{n}(\bar{\vec{\theta}})=O(\frac{1}{n})$\end{lem}
\begin{proof}
At $\bar{\vec{\theta}}$, the gradient field $\bar{h}$ corresponding
to the $T_{n}$ component is always $0$ at the stationary point so
we focus on the $\vec{f}_{n}$ part.

\begin{align*}
\vec{f}_{n}(\bar{\vec{\theta}}) & =\frac{1}{\bar{T}}\mathbb{E}_{\bar{\vec{\mu}},\bar{T}}\left[\frac{\sum_{k=0}^{\tau-1}\left(\mathbb{I}(X_{k}=\cdot)-\bar{\vec{\mu}}(\cdot)\right)}{(1+\frac{\tau}{n\bar{T}})}\mathbb{I}(\tau>\sqrt{n})\right]+\frac{1}{\bar{T}}\mathbb{E}_{\bar{\vec{\mu}},\bar{T}}\left[\frac{\sum_{k=0}^{\tau-1}\left(\mathbb{I}(X_{k}=\cdot)-\bar{\vec{\mu}}(\cdot)\right)}{(1+\frac{\tau}{n\bar{T}})}\mathbb{I}(\tau\leq\sqrt{n})\right].
\end{align*}
The first part can be bounded by Cauchy-Schwartz inequality
\begin{align*}
\frac{1}{\bar{T}}\mathbb{E}_{\bar{\vec{\mu}},\bar{T}}\left[\frac{\sum_{k=0}^{\tau-1}(\mathbb{I}(X_{k}=\cdot)-\bar{\vec{\mu}}(\cdot))}{(1+\frac{\tau}{n\bar{T}})}\mathbb{I}(\tau>\sqrt{n})\right] & \leq\frac{1}{\bar{T}}\left[\mathbb{E}_{\bar{\vec{\mu}},\bar{T}}\left|\sum_{k=0}^{\tau-1}(\mathbb{I}(X_{k}=\cdot)-\bar{\vec{\mu}}(\cdot))\right|^{2}\right]^{\frac{1}{2}}\mathbb{P}(\tau>\sqrt{n})^{\frac{1}{2}}\\
 & =O(e^{-c\sqrt{n}}).
\end{align*}

The second part in the expansion of $\vec{f}_{n}(\bar{\vec{\theta}})$
can be further expanded using Taylor polynomial

\begin{eqnarray*}
\frac{1}{\bar{T}}\mathbb{E}_{\bar{\vec{\mu}},\bar{T}}\left[\frac{\sum_{k=0}^{\tau-1}\left(\mathbb{I}(X_{k}=\cdot)-\bar{\vec{\mu}}(\cdot)\right)}{(1+\frac{\tau}{n\bar{T}})}\mathbb{I}(\tau\leq\sqrt{n})\right] & = & \swarrow\\
\frac{1}{\bar{T}}\mathbb{E}_{\bar{\vec{\mu}},\bar{T}}\left[\left(\sum_{k=0}^{\tau-1}\left(\mathbb{I}(X_{k}=\cdot)-\bar{\vec{\mu}}(\cdot)\right)\right)\left(1-\frac{\tau}{n\bar{T}}+\frac{2}{(1-c)^{3}}\frac{\tau^{2}}{n^{2}\bar{T}^{2}}\right)\mathbb{I}(\tau\le\sqrt{n})\right] & = & \swarrow\\
(1)+(2)+(3)
\end{eqnarray*}
where (1), (2), and (3) are obtained by multiplying $\sum_{k=0}^{\tau-1}\left(\mathbb{I}(X_{k}=\cdot)-\bar{\vec{\mu}}(\cdot)\right)$
through the second bracket. $c$ is a number between $0$ and $\frac{\tau}{n\bar{T}}$.
To get a bound on (1), we have

\begin{eqnarray*}
\mathbb{E}_{\bar{\vec{\mu}},\bar{T}}\left[\left(\sum_{k=0}^{\tau-1}\left(\mathbb{I}(X_{k}=\cdot)-\bar{\vec{\mu}}(\cdot)\right)\right)\mathbb{I}(\tau\leq\sqrt{n})\right] & = & \mathbb{E}_{\bar{\vec{\mu}},\bar{T}}\left[\left(\sum_{k=0}^{\tau-1}\left(\mathbb{I}(X_{k}=\cdot)-\bar{\vec{\mu}}(\cdot)\right)\right)\right]\\
 & - & \mathbb{E}_{\bar{\vec{\mu}},\bar{T}}\left[\left(\sum_{k=0}^{\tau-1}\left(\mathbb{I}(X_{k}=\cdot)-\bar{\vec{\mu}}(\cdot)\right)\right)\mathbb{I}(\tau>\sqrt{n})\right]\\
 & = & 0-O(e^{-c\sqrt{n}}).
\end{eqnarray*}

Modulus of (2) becomes bounded by
\[
\mathbb{E}_{\bar{\vec{\mu}},\bar{T}}\left|\sum_{k=0}^{\tau-1}\left(\mathbb{I}(X_{k}=\cdot)-\bar{\vec{\mu}}(\cdot)\right)\right|\frac{\tau}{n\bar{T}}\mathbb{I}(\tau\leq\sqrt{n})=O\left(\frac{1}{n}\right).
\]
Modulus of (3) becomes bounded by
\[
\mathbb{E}_{\bar{\vec{\mu}},\bar{T}}\left|\sum_{k=0}^{\tau-1}\left(\mathbb{I}(X_{k}=\cdot)-\bar{\vec{\mu}}(\cdot)\right)\right|\frac{4\tau^{2}}{n^{2}\bar{T}^{2}}\mathbb{I}(\tau\leq\sqrt{n})=O\left(\frac{1}{n}\right).
\]

\end{proof}

\subsubsection{$\sqrt{n}$-averaging}
\begin{lem}
$\lim_{n,m}\frac{1}{\sqrt{m}}\sum_{i=n}^{n+mt-1}\vec{g}_{i}(\bar{\vec{\theta}})=0$
uniformly in each small t-interval.\label{lem:-sqrtn-averaging}\end{lem}
\begin{proof}
By Lemma \ref{lem:rate-of-gn}, the expression becomes $\frac{1}{\sqrt{m}}\sum_{i=n}^{n+mt-1}O\left(\frac{1}{i}\right)=\frac{1}{\sqrt{m}}O(\log(n+mt-1)-\log n)$.
If we maximize $n$ for every $m$, we find that the log difference
is maximized when $n=1$. Hence the limit becomes bounded by 
\[
\frac{1}{\sqrt{m}}O(\log(mt))\rightarrow0
\]
uniformly on a small t-interval. 
\end{proof}

\subsubsection{Hurwitz condition}
\begin{lem}
\label{lem:Jacobian-convergence}$(D\vec{g}_{n})(\bar{\vec{\theta}})\rightarrow(D\bar{\vec{g}})(\bar{\vec{\theta}})$ \end{lem}
\begin{proof}
By Lemma \ref{lem:uniformness-of-error} we know 
\[
\vec{g}_{n}(\vec{\theta})=\vec{g}_{n}(\bar{\vec{\theta}})+(D\vec{g}_{n})(\bar{\vec{\theta}})(\vec{\theta}-\bar{\vec{\theta}})+o(|\vec{\theta}-\bar{\vec{\theta}}|).
\]
If we take the limit as $n\rightarrow\infty$ we get
\[
\bar{\vec{g}}(\vec{\theta})=0+\lim_{n}(D\vec{g}_{n})(\bar{\vec{\theta}})(\vec{\theta}-\bar{\vec{\theta}})+o(|\vec{\theta}-\bar{\vec{\theta}}|).
\]
Expand the left hand side by Taylor series and get
\begin{align*}
(D\bar{\vec{g}})(\bar{\vec{\theta}})(\vec{\theta}-\bar{\vec{\theta}}) & =\lim_{n}(D\vec{g}_{n})(\bar{\vec{\theta}})(\mbox{\ensuremath{\vec{\theta}}-\ensuremath{\bar{\vec{\theta}}}})+o(\left|\vec{\theta}-\bar{\vec{\theta}}\right|)\\
\limsup_{n}\left|(D\vec{g}_{n}-D\bar{\vec{g}})(\bar{\vec{\theta}})\right| & =\frac{o(\left|\vec{\theta}-\bar{\vec{\theta}}\right|)}{\left|\vec{\theta}-\bar{\vec{\theta}}\right|}
\end{align*}
and the right hand side is arbitrarily small so the $\limsup$ is
0.\end{proof}
\begin{thm}
\label{thm:hurwitz}Let $A\triangleq(D\bar{\vec{g}})(\bar{\vec{\theta}})$.
$(A+I/2)$ is Hurwitz when the eigenvalues of the matrix $Q$ satisfies
the condition
\[
\max_{\tilde{\lambda}\in\textrm{NPV }}Re\left(\frac{1}{1-\tilde{\lambda}}\right)<\frac{1}{2}\left(\frac{1}{1-\lambda_{PV}}\right).
\]
\end{thm}
\begin{proof}
Again let us recall from equation \ref{eq:full_ODE} that $\bar{\vec{g}}$
contains the $\bar{\vec{f}}$ component as well as the \textbf{$\bar{h}$
}component. With the notation $B\triangleq(I-Q^{T})^{-1}$, the Jacobians
are given by
\begin{align*}
\nabla_{\vec{\mu}}\bar{\vec{f}}(\vec{\mu},T) & =\frac{1}{T}\left[B-(\vec{1}^{\tp}B\vec{\mu})I-\vec{\mu}\vec{1}^{\tp}B\right]\\
\nabla_{T}\bar{\vec{f}}(\vec{\mu},T) & =-\frac{1}{T^{2}}\left[B\vec{\mu}-(\vec{1}^{\tp}B\vec{\mu})\vec{\mu}\right]\\
\nabla_{\vec{\mu}}\bar{h}(\vec{\mu},T) & =\vec{1}^{\tp}B\\
\nabla_{T}\bar{h}(\vec{\mu},T) & =-1\quad\textrm{(scalar)}.
\end{align*}
At the stationary point $(\bar{\vec{\mu}},\bar{T})$, if we define
$\beta=\vec{1}^{\tp}B\bar{\vec{\mu}}=\bar{T}$, the $\bar{\vec{f}}$
component becomes
\begin{align*}
\nabla_{\vec{\mu}}\bar{\vec{f}}(\overline{\vec{\mu}},\bar{T}) & =\frac{1}{\beta}\left[B-\beta I-\bar{\vec{\mu}}\vec{1}^{\tp}B\right]\quad\textrm{call this matrix J}\\
\nabla_{T}\bar{\vec{f}}(\bar{\vec{\mu}},\bar{T}) & =0.
\end{align*}

We will now established a 1-1 correspondence between the eigenvectors
of $J$ and the eigenvectors of $B$. The overall Jacobian would,
in block form, look like
\[
\left[\begin{array}{cc}
J & \vec{0}^{\tp}\\
\vec{1}^{\tp}B & -1
\end{array}\right].
\]
This has the same eigenvalues as $J$ with the addition of the eigenvalue
-1. That would not violate the Hurwitz condition. Hence we need to
ensure that $J+\frac{I}{2}$ is Hurwitz. 

Given a vector $\vec{y}$ such that $J\vec{y}=\lambda_{J}\vec{y}$
and $\vec{y}$ linearly independent of $\bar{\vec{\mu}}$. Define
$\vec{x}\triangleq\vec{y}+r\bar{\vec{\mu}}$. That means $B\vec{x}=\beta(\lambda_{J}+1)\vec{y}+(r\beta+\vec{1}^{\tp}B\vec{y})\bar{\vec{\mu}}$.
The correct $r$ that would make $\vec{x}$ an eigenvector of $B$
is such that $r\beta\lambda_{J}=\vec{1}^{\tp}B\vec{y}$. Here $\beta$
is the principal eigenvalue of $B$ so it is strictly positive. That
means there exists such $r$ if $\lambda_{J}\neq0$. The corresponding
eigenvalue for $B$ would be $\lambda_{B}\triangleq\beta(\lambda_{J}+1)$.
If $\vec{y}$ is a multiple of $\bar{\vec{\mu}}$, then its $J$ eigenvalue
would be $-1$ and its $B$ eigenvalue would be $\beta$. Below in
Lemma \ref{lem:eigenvalue-not-zero}, we show that $\lambda_{J}$
can never be $0$. This would imply every eigenvector of $J$ is an
eigenvector of $B$.

Conversely, given a vector $\vec{z}$, $B\vec{z}=\lambda_{B}\vec{z}$,
we can define $\vec{u}\triangleq\vec{z}+r\bar{\vec{\mu}}$. If we
choose $r=\frac{\beta+\lambda_{B}(\vec{1}^{\tp}\vec{z})}{\beta-\lambda_{B}}$
then $J\vec{u}=\left(\frac{\lambda_{B}}{\beta}-1\right)\vec{u}$.
This would work when $\vec{z}$ is not the principal right-eigenvector
of $B$. If it is, i.e. $\vec{z}=\bar{\vec{\mu}}$, then trivially
$J\bar{\vec{z}}=-\bar{\vec{z}}$.

Hence we can conclude that there is a one-to-one correspondence between
the eigenvector/eigenvalues of $J$ and $B$ and the relation is given
by
\begin{align*}
\lambda_{J} & =\frac{\lambda_{B}}{\beta}-1\quad\textrm{for \ensuremath{\lambda_{B}\neq\beta}or \ensuremath{\lambda_{J}\neq0}}\\
\lambda_{J}=-1 & \textrm{ when \ensuremath{\lambda_{B}=\beta}}.
\end{align*}
Hence, in order to ensure that $J+\frac{I}{2}$ is Hurwitz, we require
\begin{align*}
Re(\frac{\lambda_{B}}{\beta}-\frac{1}{2}) & <0\quad\forall\lambda_{B}\neq\beta\\
\Rightarrow Re(\lambda_{B}) & <\frac{\beta}{2}\quad\forall\lambda_{B}\neq\beta\\
\Rightarrow\max_{\tilde{\lambda}\in\textrm{NPV }} & Re\left(\frac{1}{1-\tilde{\lambda}}\right)<\frac{1}{2}\left(\frac{1}{1-\lambda_{PV}}\right).
\end{align*}
\end{proof}
\begin{rem}
If you carefully examine the proof in \cite{kushner2003stochastic},
you will notice that the Jacobian $J$ is drift of an Ornstein-Uhlenbeck
process $U(t)$ that lives on the subspace orthogonal to $\vec{1}$.
Hence, in order for the OU process to have a stationary distribution,
it is enough to require that all the eigenvectors that live on this
hyperplane have real part of their eigenvalue less than $\frac{1}{2}$.
If we are given an eigenvector $J\vec{y}=\lambda\vec{y}$, we can
dot this with $\vec{1}$ and arrive at
\begin{align*}
J\vec{y} & =\lambda\vec{y}\\
\Rightarrow-\vec{1}^{\tp}\vec{y} & =\lambda\vec{1}^{\tp}\vec{y}.
\end{align*}
This implies that if $\lambda\neq-1,$ then $\vec{1}^{\tp}\vec{y}=0$
which is in $\vec{y}\in\vec{1}^{\bot}$ so it is a relevant eigenvector.
If $\lambda=-1$, it would not affect the Hurwitz condition anyways.
So our sufficient condition above is not overly strong.\end{rem}
\begin{lem}
\label{lem:eigenvalue-not-zero}$\lambda_{J}\neq0$\end{lem}
\begin{proof}
Assume there exists a $\vec{y}$ such that $J\vec{y}=0$. That means
\begin{align*}
B\vec{y} & =\beta\vec{y}+\bar{\vec{\mu}}\vec{1}^{\tp}B\vec{y}\\
(I-\bar{\mu}\vec{1}^{\tp})B\vec{y} & =\beta\vec{y}.
\end{align*}
We recognize that $P\triangleq(I-\bar{\vec{\mu}}\vec{1}^{\tp})$ is
a (non-orthogonal) projection. Also $P\bar{\vec{\mu}}=\vec{0}$ and
$\vec{1}^{\tp}P=\vec{0}$. This means $\beta$ is an eigenvalue of
$PB$, that means there would exists a left eigenvector $\vec{x}$
such that
\begin{align*}
\vec{x}^{\tp}PB & =\beta\vec{x}^{\tp}\\
\vec{x}^{\tp}P & =\beta\vec{x}^{\tp}B^{-1}.
\end{align*}
We decouple this into two equations
\begin{align*}
\vec{x}^{\tp}P & =\vec{z}\\
\vec{z}^{\tp} & =\beta\vec{x}^{\tp}B^{-1}.
\end{align*}
If the equation $\vec{x}^{\tp}P=\vec{z}^{\tp}$ has a solution $x$,
we must require $\vec{z}^{\tp}\bar{\vec{\mu}}=0$. That would mean
$\vec{x}$ can be decomposed as a fundamental solution added to a
null solution. The null space is $c\vec{1}$ and $\vec{z}^{\tp}P=\vec{z}^{\tp}-\vec{z}^{\tp}\bar{\vec{\mu}}\vec{1}^{\tp}=\vec{z}^{\tp}$.
So $\vec{x}^{\tp}=c\vec{1}^{\tp}+\vec{z}^{\tp}$ would span the entire
solution space. However, remember that we're interested in $\vec{z}^{\tp}=\beta\vec{x}^{\tp}B^{-1}$.
We dot this with $\bar{\vec{\mu}}$ and arrive at
\begin{align*}
0 & =\beta\vec{x}^{\tp}(I-Q^{\tp})\bar{\vec{\mu}}\\
0 & =\beta(\vec{x}^{\tp}\bar{\vec{\mu}}-\lambda\vec{x}^{\tp}\bar{\vec{\mu}})\\
\vec{x}^{\tp}\bar{\vec{\mu}} & =\lambda\vec{x}^{\tp}\bar{\vec{\mu}}\\
\vec{x}^{\tp}\bar{\vec{\mu}} & =0\quad\textrm{because \ensuremath{0<\lambda<1}}.
\end{align*}
Now we can conclude that $c=0$:
\begin{align*}
\vec{x}^{\tp} & =c\vec{1}^{\tp}+\vec{z}^{\tp}\\
\vec{x}^{\tp}\bar{\vec{\mu}} & =c\vec{1}^{\tp}\bar{\vec{\mu}}+\vec{z}^{\tp}\bar{\vec{\mu}}\\
0 & =c+0.
\end{align*}
This means $\vec{x}^{\tp}=\vec{z}^{\tp}=\beta\vec{x}^{\tp}B^{-1}\Rightarrow\vec{x}^{\tp}B=\beta\vec{x}^{\tp}$.
This would mean $\vec{x}^{\tp}$ is the principle left-eigenvector
and all its components are strictly positive. In that case, it would
be impossible to have $\vec{x}^{\tp}\bar{\vec{\mu}}=0$. So there
is no eigenvalue $\beta$ for the matrix $PB$. Hence $J$ cannot
have a zero eigenvalue.
\end{proof}

\subsubsection{Quadratic variation of the martingales}
\begin{lem}
\label{lem:quadratic-variation}Define $\delta\vec{M}_{n}=\vec{W}_{n}-\mathbb{E}_{n}(\vec{W}_{n})$.
There exists a $p>0$ such that
\[
\sup_{n}\mathbb{E}|\delta\vec{M}_{n}|^{2+p}<\infty
\]
and a non-negative definite matrix $\Sigma$ such that
\[
\mathbb{E}_{n}\delta\vec{M}_{n}\delta\vec{M}_{n}^{\tp}\rightarrow\Sigma.
\]
\end{lem}
\begin{proof}
Recall that $\vec{W}_{n}=(\vec{Y}_{n},Z_{n})$ in Equation \ref{eq:mainalg_formal}.
Pick $p=2$, we can use Jensen's inequality and arrive at
\[
|\delta\vec{M}_{n}|^{4}\leq2\left(|\vec{Y}_{n}-\mathbb{E}_{n}(\vec{Y}_{n})|^{4}+|Z_{n}-\mathbb{E}_{n}(Z_{n})|^{4}\right)\leq16\left(|\vec{Y}_{n}|^{4}+|\mathbb{E}_{n}(\vec{Y}_{n})|^{4}+|Z_{n}|^{4}+|\mathbb{E}_{n}Z_{n}|^{4}\right).
\]
Due to the facts\end{proof}
\begin{enumerate}
\item $|\vec{Y}_{n}|\leq\tau(\vec{\mu}_{n})$, here $\tau(\vec{\mu}_{n})$
is the stopping time given the initial distribution $\mu_{n}$,

\begin{enumerate}
\item $\mathbb{E}_{n}|\vec{Y}_{n}|^{4}\leq\mathbb{E}_{n}(|\vec{Y}_{n}|^{4})$,
\item $|Z_{n}|\leq\tau(\vec{\mu}_{n})$,
\item $\mathbb{E}_{n}(Z_{n})^{4}\leq\mathbb{E}_{n}(Z_{n}^{4})$,
\item $\tau(\vec{\mu})\leq_{a.s.}\tilde{\tau}\quad\forall\vec{\mu}$ where
$\tilde{\tau}=\max\{\tau(\vec{\mu})|\vec{\mu}=\vec{\delta}_{x},\forall x\in S\}$
by stochastic dominance followed by Skorohod representation,
\end{enumerate}
\end{enumerate}
\begin{proof}
we can conclude that $\sup_{n}\mathbb{E}|\delta\vec{M}_{n}|^{4}<\infty$.

We now use dominated theorem on $\mathbb{E}_{n}\delta\vec{M}_{n}\delta\vec{M}_{n}^{\tp}$.
We can think in block matrix form 
\[
\delta\vec{M}_{n}\delta\vec{M}_{n}^{\tp}=\left[\begin{array}{cc}
(\vec{Y}_{n}(\vec{\theta}_{n})-\vec{f}_{n}(\vec{\theta}_{n}))(\vec{Y}_{n}(\vec{\theta}_{n})-\vec{f}_{n}(\vec{\theta}_{n}))^{\tp} & (\vec{Y}_{n}(\vec{\theta}_{n})-\vec{f}(\vec{\theta}_{n}))(Z_{n}(\vec{\theta}_{n})-h_{n}(\vec{\theta}_{n}))^{\tp}\\
(Z_{n}(\vec{\theta}_{n})-h_{n}(\vec{\theta}_{n}))(\vec{Y}_{n}(\vec{\theta}_{n})-\vec{f}_{n}(\vec{\theta}_{n}))^{\tp} & (Z_{n}(\vec{\theta}_{n})-h_{n}(\vec{\theta}_{n}))(Z_{n}(\vec{\theta}_{n})-h_{n}(\vec{\theta}_{n}))^{\tp}
\end{array}\right].
\]
In absolute value conditioned on $\mathscr{F}_{n}$, each entry of
this matrix is dominated by $2\tau(\vec{\theta}_{n})\leq2\tilde{\tau}$.
$\delta\vec{M}_{n}\delta\vec{M}_{n}^{\tp}$ is dominated entry-wise
by $2\tilde{\tau}$ uniformly for all possible admissible $\vec{\theta}_{n}$.
With a few more steps, we can show the convergence to a non-negative
matrix.

$\vec{f}_{n}(\vec{\theta}_{n})\rightarrow\vec{0}$ and $h_{n}(\vec{\theta}_{n})\rightarrow0$
by Lemma \ref{lem:rate-of-gn}.

$Z_{n}(\vec{\theta})=\tau(\vec{\theta})-T$ can be represented in
a way that is continuous in $\vec{\theta}$ (by writing $\tau$ as
a mixture of the initial starting points). Hence $Z_{n}(\vec{\theta}_{n})\rightarrow Z(\bar{\vec{\theta}})=\tau(\bar{\vec{\theta}})-\bar{T}$.

$\vec{Y}_{n}(\vec{\theta})=\frac{\sum_{k=0}^{\tau(\vec{\theta})-1}\left(\mathbb{I}(X_{k}=\cdot|X_{0}\sim\vec{\mu})-\vec{\mu}\right)}{T+\frac{\tau}{n+1}}$
can also be written in a way that is continuous in $\vec{\theta}$
and uniformly convergent (over $\vec{\theta})$ the random variable
$\vec{Y}(\vec{\theta})$. Hence 
\[
\vec{Y}_{n}(\vec{\theta}_{n})\rightarrow_{n\rightarrow\infty}\frac{\sum_{k=0}^{\tau(\bar{\vec{\mu}})-1}\left(\mathbb{I}(X_{k}=\cdot|X_{0}\sim\bar{\vec{\mu}})-\bar{\vec{\mu}}\right)}{\bar{T}}\triangleq\vec{Y}(\bar{\vec{\theta}}).
\]

$\left(\begin{array}{c}
\vec{Y}(\bar{\vec{\theta}})\\
Z(\bar{\vec{\theta}})
\end{array}\right)\left(\begin{array}{cc}
\vec{Y}(\bar{\vec{\theta}}) & Z(\bar{\vec{\theta}})\end{array}\right)$ can also be shown to be entry-wise dominated, hence its expected
value is well defined. It is obviously a non-negative definite matrix
because of the form $\vec{x}\vec{x}^{\tp}$.

Together with dominated convergence and the fact that $L_{1}$-convergence
implies convergence in probability, we have the conclusion where $\Sigma=\mathbb{E}\left[\left(\begin{array}{c}
\vec{Y}(\bar{\vec{\theta}})\\
Z(\bar{\vec{\theta}})
\end{array}\right)\left(\begin{array}{cc}
\vec{Y}(\bar{\vec{\theta}}) & Z(\bar{\vec{\theta}})\end{array}\right).\right]$
\end{proof}

\subsection{Continuous-Time version proof}

The ODE associated with the continuous-time algorithm can be arrived
by similar technique as the discrete-time case.

\begin{eqnarray*}
\dot{\vec{\mu}}(t) & = & -\frac{1}{T(t)}\left(\vec{\mu}(t)^{\tp}Q^{-1}-(\vec{\mu}(t)^{\tp}Q^{-1}\vec{1})\vec{\mu}(t)^{\tp}\right)\\
\dot{T}(t) & = & -\vec{\mu}(t)^{\tp}Q^{-1}\vec{1}-T(t).
\end{eqnarray*}
Note that instead of $(I-Q)^{-1}$ appearing we now have $-Q^{-1}$
where $Q$ is a transition rate matrix. The Perron-Frobenius theorem
still applies to matrix of the form $M-D$ where $M$ consists of
off-diagonal non-negative entries and $D$ is a non-positive diagonal
matrix. In the discrete-time proof, we often used the fact that $\exp(I-Q)^{-1}$
is a matrix of non-negative entries. We need to now show that $\exp(-Q)^{-1}$
is also a matrix of non-negative entries.
\begin{lem}
\label{lem:nonnegative-entries}Given a transition rate matrix $Q$,
$\exp(-Q)^{-1}$ is a matrix of non-negative entries\end{lem}
\begin{proof}
Let $\alpha$ be the maximum of the diagonal element of $-Q$. Notice
that $-Q^{-1}=\alpha\left[I-\left(\frac{1}{\alpha}Q+I\right)\right]^{-1}=\alpha\sum_{k\geq0}(\frac{1}{\alpha}Q+I)^{k}$
for $\alpha$ large enough. $\frac{1}{\alpha}Q+I$ is a non-negative
irreducible matrix hence the infinite sum is also a positive matrix.\end{proof}
\begin{lem}
The principal eigenvalue of $Q$ is real and smaller than $0$, and
also the left and right principal eigenspaces are one-dimensional.\end{lem}
\begin{proof}
The spectral radius is bounded above by $\parallel Q\parallel_{\infty}\leq0$.
$Q+cI$ can be made into a irreducible non-negative matrix for some
large $c$. All the properties follow after applying Perron-Frobenius
to that.
\end{proof}
Analogous to Lemma \ref{lem:-Tn-boundedness}, we need to show that
$T_{n}$ is bounded above and below almost surely. 
\begin{lem}
The $T_{n}$ sequence in the continuous-time case is bounded above
and below almost surely by finite random variables.\end{lem}
\begin{proof}
The upper bound is identical to the discrete-time case (Lemma \ref{lem:-Tn-boundedness}).
For the lower bound, we consider a random variable $\tilde{s}$ that
is an exponential rate corresponding to the slowest holding rate of
the Markov chain. This way, $\tilde{s}^{(k)}$ can be coupled to the
first holding time of the Markov chain during the k-th iteration.
This way, $\tau^{(k)}\geq_{a.s.}\tilde{s}^{(k)}$. Hence $\lim_{n}T_{n}\geq\mathbb{E}(\tilde{s})$
implying that its bounded below by a finite random variable.
\end{proof}

\section*{Figures}

\begin{figure}[H]
\centerline{\includegraphics[width=7in]{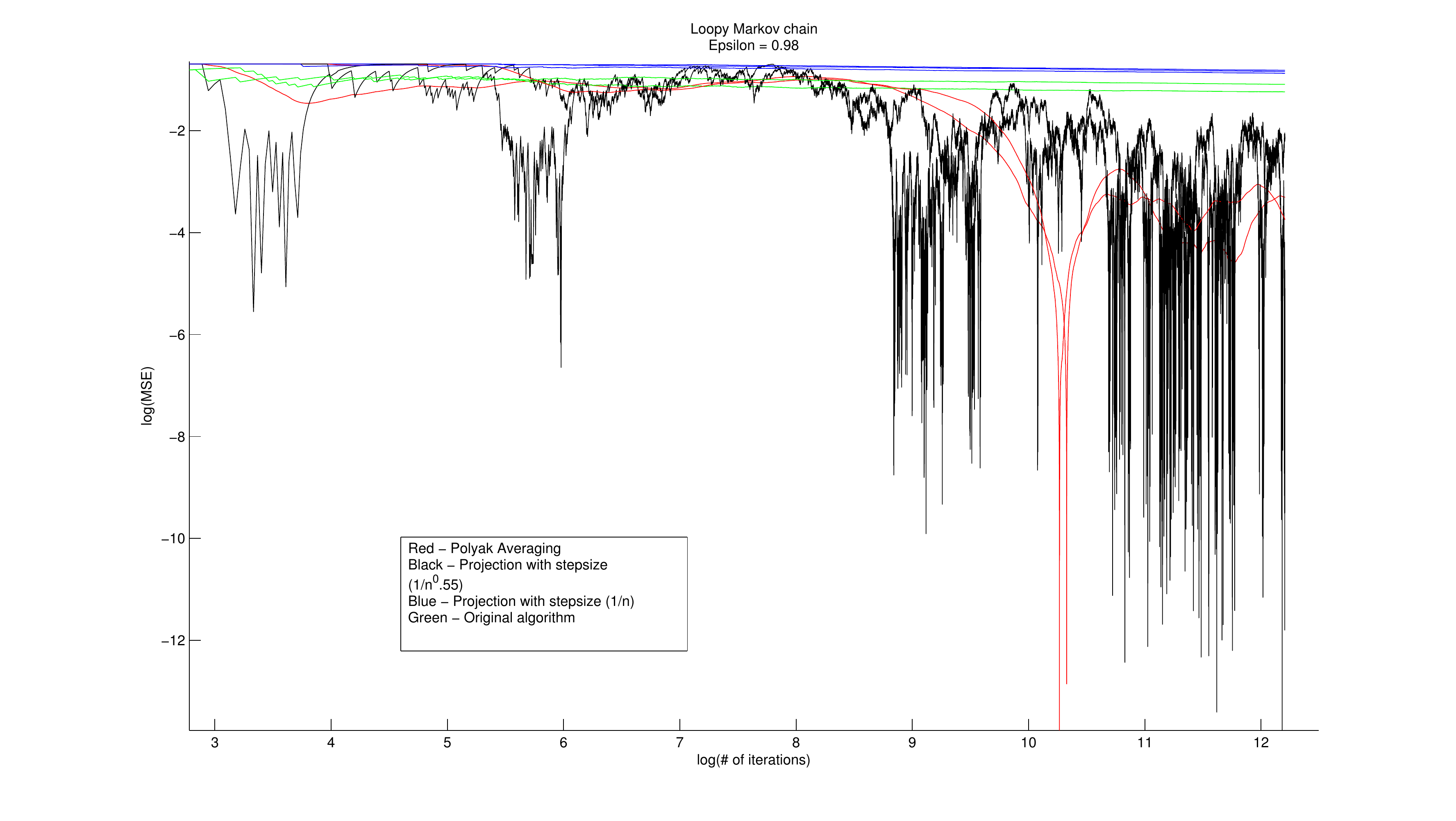}}

\caption{\label{fig:loopyMC}This figure is the time vs. error plot of the
main algorithm ran on a loopy Markov chain with eigenvalues well outside
the CLT regime ($\epsilon=0.98>0.5$). The plot is a log/log plot
where the y-axis is the Mean-Squared-Error. It is clear that the Polyak-Ruppert
Averaging (the red line) converges much faster than the original algorithm
(green line)}

\end{figure}
 
\begin{figure}[H]
\includegraphics[width=0.7\paperwidth]{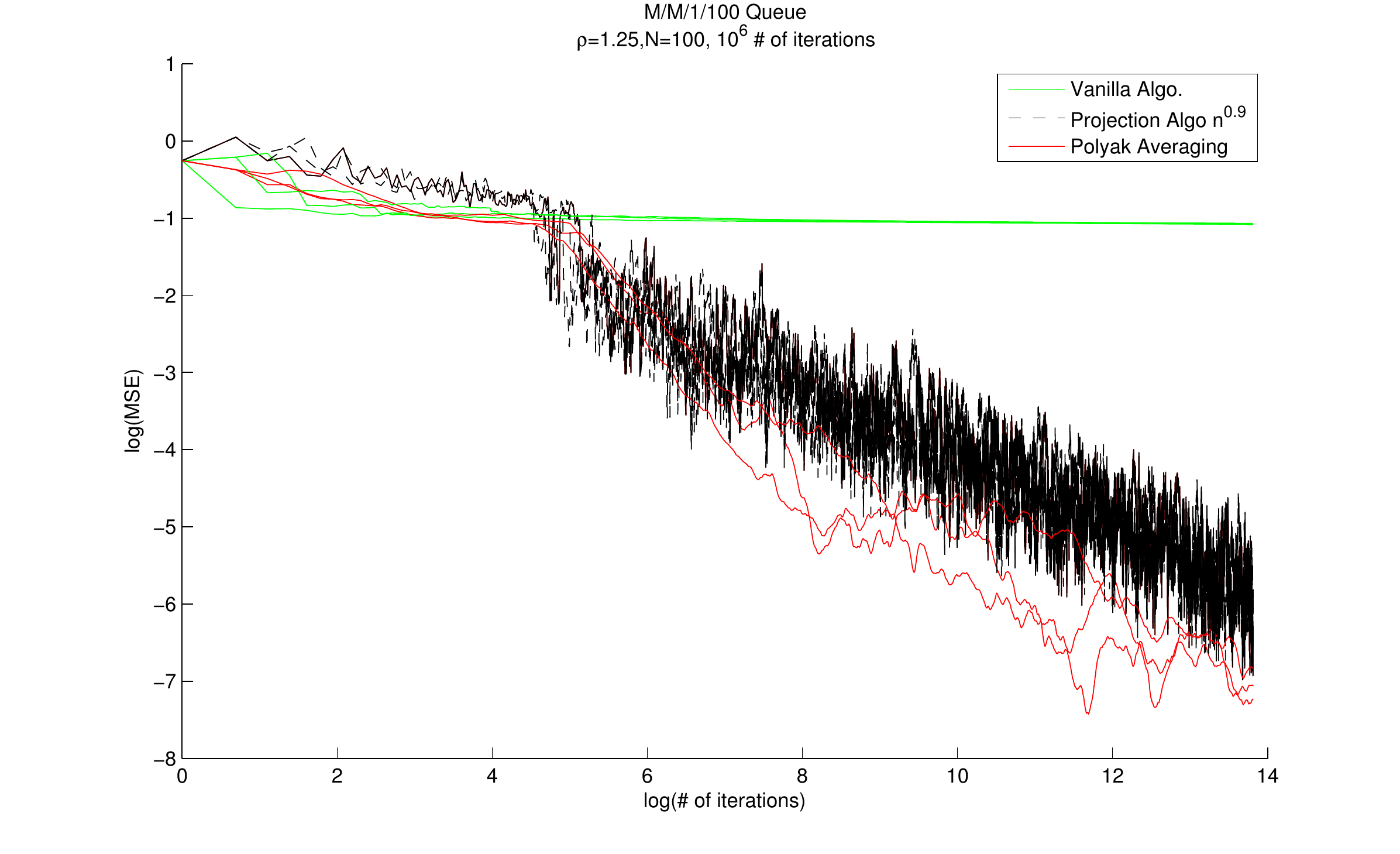}\caption{\label{fig:MM1}This is the simulation of a M/M/1 queue with 100 queue
capacity and $\rho=1.25$. We are considering the embedded discrete-time
chain at the jump times of the system. We had to Doeblinize the process
(multiply transition matrix by 0.95) in order to deal with the large
$E[\tau]$ due to the system being in heavy-traffic regime. As you
can see, the Polyak-Ruppert averaging (red) is significantly better
than the original algorithm (green) on the log-log plot. The eigenvalue
condition for the CLT is not satisfied after Doeblinization.}

\end{figure}

\begin{figure}[H]
\centerline{\includegraphics[width=1\paperwidth]{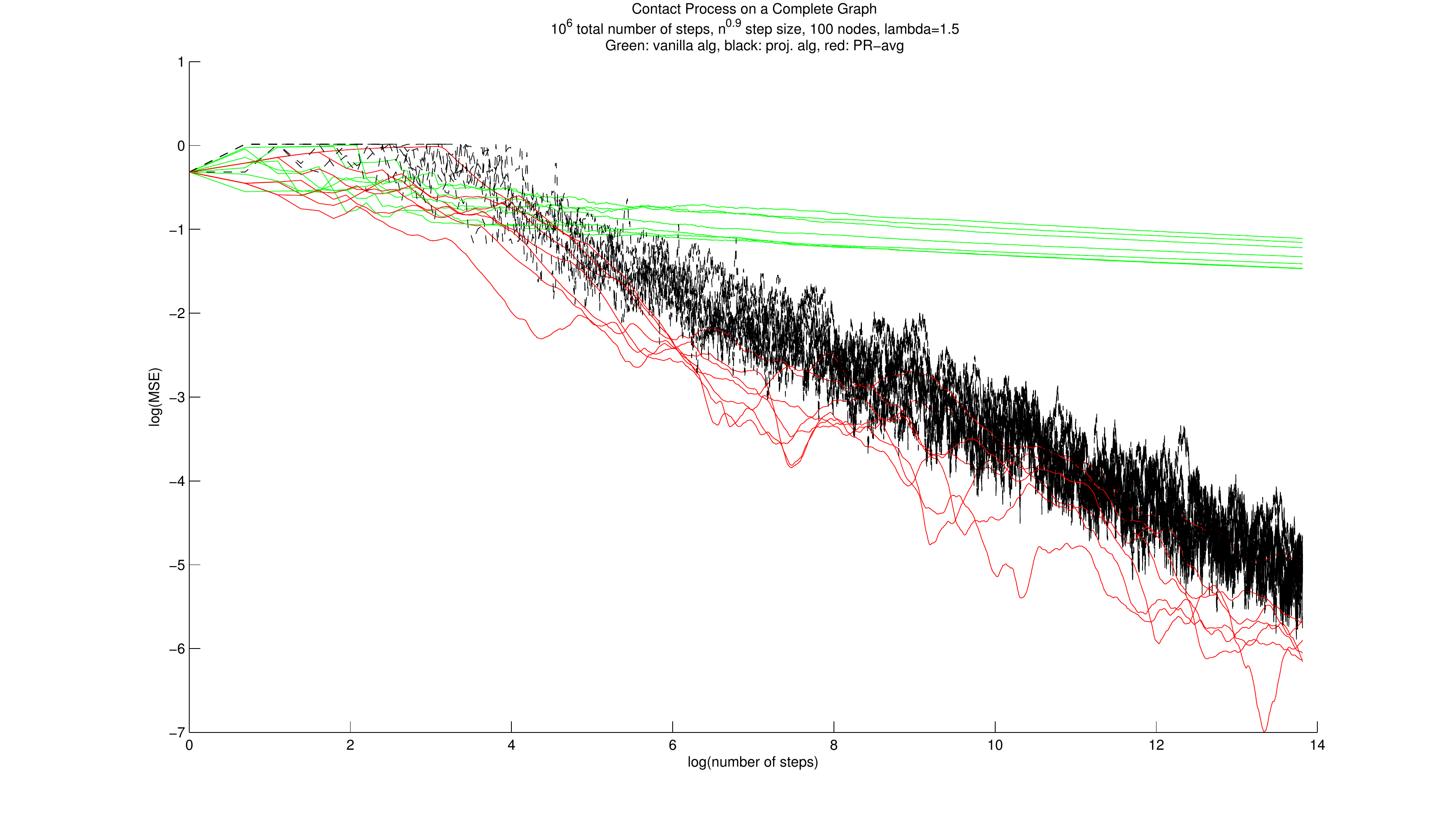}}\caption{\label{fig:cpcg}This is a simulation of the contact process on a
complete graph where $\lambda=1.5$ with 100 nodes. The plot is the
log-log plot of the number of steps vs. MSE. The sufficient condition
for CLT cannot be met in this case after subtracting $0.5I$ from
the rate matrix. The Polyak's averaging algorithm (red) significantly
outperforms the vanilla algorithm (green).}
\end{figure}

\pagebreak{}

\bibliographystyle{plain}
\bibliography{stochasticapproximations}

\end{document}